\documentclass[11pt]{article}

\usepackage[utf8]{inputenc}
\usepackage{centernot}
\usepackage[T1]{fontenc}
\usepackage{amsmath}
\usepackage{amsfonts}
\usepackage{amssymb}
\usepackage{amsthm}
\usepackage{mathtools}
\usepackage{hyperref}
\usepackage{url}
\usepackage[top=2cm, bottom=2cm, left=2.5cm, right=2.5cm]{geometry}
\newtheorem{theorem}{Theorem}[section]
\newtheorem{lemma}[theorem]{Lemma}

\newtheorem{corollary}[theorem]{Corollary}

\theoremstyle{definition}
\newtheorem{defn}{Definition}[section]

\theoremstyle{remark}
\newtheorem{remark}{Remark}[section]

\newcommand\Leb{\operatorname{Leb}}

\newcommand\Base{\operatorname{Base}}

\newcommand\Diam{\operatorname{Diam}}

\allowdisplaybreaks

\begin{document}
\hypersetup{pageanchor=false}
\title{Supercritical loop percolation on $\mathbb{Z}^d$ for $d\geq 3$} 

\author{Yinshan Chang\footnote{Max-Planck Institute in Science, 04103, Leipzig, Germany. Email: ychang@mis.mpg.de}}
\maketitle
\begin{abstract}
 The loop cluster model was introduced by Y. Le Jan in \cite{LeJanMR2971372}, which is a model of random graphs constructed from a Poisson point process $\mathcal{L}_{\alpha}$ of loops on countable graphs. Two vertices are in the same cluster if they are connected through a sequence of intersecting loops. In this paper, we are interested in the loop cluster model on $\mathbb{Z}^d$ for $d\geq 3$. It is a long range model with two parameters $\alpha$ and $\kappa$, where the non-negative parameter $\alpha$ measures the amount of loops, and $\kappa$ plays the role of killing on vertices penalizing ($\kappa\geq 0$) or favoring ($\kappa<0$) appearance of large loops. We consider the truncated loop cluster model formed by the Poisson point process $\mathcal{L}_{\alpha,\leq m}$, which is the restriction of $\mathcal{L}_{\alpha}$ on loops with at most $m$ jumps. We prove the existence of percolation in a $2$-dimensional slab for the truncated loop model $\mathcal{L}_{\alpha,\leq m}$ as long as the intensity parameter $\alpha$ is strictly above the critical threshold of the non-truncated loop model and $m$ is large enough. We apply this result to prove the exponential decay of one arm connectivity for the finite cluster at $0$ for the whole supercritical regime of the non-truncated loop model. For $\kappa=0$, this loop percolation model provides an example in which we have different behaviors of finite clusters in sub-critical and super-critical regimes. Also, we deduce the strict increase of the critical curve $\alpha\rightarrow\kappa_c(\alpha)$ for $\alpha\geq\alpha_c$, where $\alpha_c$ is the critical value when $\kappa=0$. In the end, we prove that $\forall\alpha>\alpha_c$ large balls in the infinite cluster are finally very regular in the sense of \cite{Sapozhnikov2014}, which implies that large balls are finally very good in the sense of \cite{BarlowMR2094438}. By \cite{BarlowMR2094438} and \cite{BarlowHamblyMR2471657}, we have Harnack's inequality and Gaussian type estimate for simple random walks on the infinite cluster for all $\alpha>\alpha_c$.

\end{abstract}

\section{Introduction}
 The loop cluster model is a model of random graphs constructed from a loop soup (a Poisson point process of loops) on a finite or countable graph. It was introduced by Y. Le Jan in \cite{LeJanMR2971372} and studied by S. Lemaire and Le Jan in \cite{LeJanLemaireMR3263044}, by A. Sapozhnikov and the author in \cite{loop-perc-Zd}, by T. Lupu in \cite{loop-perc-gff-coupling}, \cite{loop-perc-H} and \cite{loop-perc-H-cvg}, by F. Camia in \cite{loop-perc-big-cvg}. Also, note that the Brownian loop soup clusters have already been studied in the context of CLE by S. Sheffield and W. Werner in \cite{SheffieldWernerMR2979861}.

 We adopt the same notation as in \cite{loop-perc-Zd}. Consider an unweighted undirected graph $G=(V,E)$ and a random walk $(X_m,m\geq 0)$ on it with transition matrix $Q$. Unless specified, we will assume that $(X_m,m\geq 0)$ is a simple random walk (SRW) on $\mathbb Z^d$.
 As in \cite{LeJanLemaireMR3263044}, an element $\dot\ell = (x_1,\dots,x_n)$ of $V^n$, $n\geq 2$, satisfying $x_1\neq x_2,\dots,x_n\neq x_1$ is called a non-trivial discrete based loop. We define its \emph{length} $|\dot{\ell}|$ to be $n$. Two based loops of length $n$ are equivalent if they coincide after a circular permutation of their coefficients, i.e. $(x_1,\dots,x_n)$ is equivalent to $(x_i,\dots,x_n,x_1,\dots,x_{i-1})$ for all $i$. Equivalence classes of non-trivial discrete based loops for this equivalence relation are called (non-trivial) discrete loops. For a loop $\ell$ (equivalence class of $\dot{\ell}$), we define its \emph{length} $|\ell|$ to be $|\dot{\ell}|$.  

Given an additional parameter $\kappa>-1$, we associate to each based loop $\dot\ell=(x_1,\dots,x_n)$ the weight
\begin{equation*}
\dot\mu_\kappa(\dot\ell)=\frac{1}{n}\left(\frac{1}{1+\kappa}\right)^{n}Q^{x_1}_{x_2}\cdots Q^{x_{n-1}}_{x_n}Q^{x_n}_{x_1}.
\end{equation*}
The push-forward of $\dot\mu_\kappa$ on the space of discrete loops is denoted by $\mu_\kappa$.

For $\alpha>0$ and $\kappa>-1$, let $\mathcal L_{\alpha,\kappa}$ be the Poisson loop ensemble of intensity $\alpha\mu_\kappa$, 
i.e, $\mathcal L_{\alpha,\kappa}$ is a random countable collection of discrete loops such that the point measure $\sum\limits_{\ell\in \mathcal{L}_{\alpha,\kappa}}\delta_{\ell}$ is 
a Poisson random measure of intensity $\alpha\mu_\kappa$. (Here, $\delta_\ell$ means the Dirac mass at the loop $\ell$ and $\mathcal{L}_{\alpha,\kappa}$ is a multi-set.) We identify $\mathcal L_{\alpha,\kappa}$ with the random measure $\sum\limits_{\ell\in \mathcal{L}_{\alpha,\kappa}}\delta_{\ell}$. The collection $\mathcal L_{\alpha,\kappa}$ is induced by the Poisson ensemble of non-trivial continuous loops defined by Le Jan \cite{LeJanMR2815763}. For $\kappa>-1$, let $\mathcal{LP}_{\kappa}=\sum\limits_{i}\delta_{(\alpha_i,\ell_i)}$ be the Poisson point process of the intensity measure $\Leb([0,\infty[)\otimes\mu_{\kappa}$, where $\Leb([0,\infty[)$ is the Lebesgue measure on $[0,\infty[$. We identify $\mathcal{LP}_{\kappa}$ with its support: $\mathcal{LP}_{\kappa}=\{(\alpha_i,\ell_i):\mathcal{LP}(\alpha_i,\ell_i)>0\}$. Then, $\mathcal{L}_{\alpha,\kappa}=\sum\limits_{(\alpha_i,\ell_i)\in\mathcal{LP}_{\kappa},\alpha_i\leq\alpha}\delta_{\ell_i}$. Similarly, for based loops, we use notation $\dot{\mathcal{L}}_{\alpha,\kappa}$ and $\dot{\mathcal{LP}}_{\kappa}$.

An edge $\{x,y\}$ is called open at time $\alpha$ if it is crossed by at least one loop $\ell\in\mathcal{L}_{\alpha,\kappa}$ in any direction. Open edges form clusters of vertices $\mathcal{C}_{\alpha,\kappa}$. For a vertex $x$, let $\mathcal{C}_{\alpha,\kappa}(x)$ be the open cluster containing $x$. Let $\alpha_c(\kappa)=\inf\{\alpha>0:\#\mathcal{C}_{\alpha,\kappa}(0)=\infty\}$ be the \emph{critical threshold} of the loop percolation. Let $\mathcal{L}_{\alpha,\kappa}^{\leq m}=\{\ell\in\mathcal{L}_{\alpha,\kappa}:|\ell|\leq m\}$. Let $\alpha_c^{(m)}(\kappa)$ be the critical threshold for the percolation on $\mathbb{Z}^d$ by $\mathcal{L}_{\alpha,\kappa}^{\leq m}$ and by $\widetilde{\alpha}_c^{(m)}(\kappa)$ the critical threshold for the percolation on $\mathrm{Slab}(m)\overset{\mathrm{def}}{=}\mathbb{Z}_{+}^{2}\times\{0,1,\ldots,m\}^{d-2}$ by $\{\ell\in\mathcal{L}_{\alpha,\kappa}^{\leq m}:\ell\subset\mathrm{Slab}(m)\}$. 

\emph{For simplicity, throughout the paper, we omit $\kappa$ in the notation if $\kappa=0$, e.g.,  $\alpha_c$ is short for $\alpha_c(0)$.}

We are particularly interested in the supercritical phase of the loop percolation. For Bernoulli bond or site percolation, ``slab percolation'' are quite useful in the study of supercritical phase. Analogously, we consider truncated loop percolation models.
\begin{theorem}\label{thm: truncated loop percolation}
 For $d\geq 3$ and $\kappa\geq 0$,
 \begin{itemize}
  \item[i)] $\alpha_c^{(2m+2)}(\kappa)<\alpha_c^{(2m)}(\kappa)$ for all $m\geq 1$,
  \item[ii)] $(\widetilde{\alpha}_c^{(2m)})_m(\kappa)\downarrow\alpha_c(\kappa)$ and $(\alpha_c^{(2m)}(\kappa))_m\downarrow\alpha_c(\kappa)$ as $m\uparrow\infty$.
 \end{itemize}
 (Note that $|\ell|$ must be even for a loop $\ell$ on $\mathbb{Z}^d$.)
\end{theorem}
Such approximation by loop percolation according to diameters of loops and the strict increase of the threshold are established by T. Lupu in the context of two dimensional SRW loop soup in \cite[Theorem~3,Proposition~4.1]{loop-perc-H}. The proof of the first part is an analogue of \cite[Lemma~3.5]{GrimmettMR1707339}, based on a comparison of derivatives, see Lemma~\ref{lem: derivative compare}. The proof of the second part is a modification of the argument for Bernoulli bond percolation (see \cite[Theorem~7.2]{GrimmettMR1707339}). We use the same renormalization schema, which is based on a stochastic comparison with Bernoulli site percolation. The key step is the sprinkling lemma (Lemma~\ref{lem: sprinkling}), which states that certain connection events appear with high probabilities after we increase \emph{locally} the intensity parameter. It is an analogue of the sprinkling lemma (\cite[Lemma~7.17]{GrimmettMR1707339}) for Bernoulli bond percolation. The main difficulty appears when $\kappa=0$. There is a long range correlation of at most polynomial decay which requires additional arguments. Our variant of sprinkling lemma is given in Section 4. For its proof, we use the independence in the loop soup and an upper bound of $\mu(\ell:\ell\cap A\neq\emptyset,\ell\cap B\neq\emptyset)$ for two disjoint vertex sets $A$ and $B$ (see \cite[Lemma~2.7 a)]{loop-perc-Zd}). 

We expect that certain quantities of the non-truncated models can be approximated by the truncated models. However, Theorem~\ref{thm: truncated loop percolation} ii) is not trivially true as certain threshold fails to converge, say $\alpha_{\#}$, corresponding to the finiteness of expected size of clusters. Indeed, by following the argument of M. Aizenman and D. J. Barsky (\cite{AizenmanBarskyMR874906}), the threshold $\alpha_{\#}^{(m)}$ for the truncated model coincides with $\alpha_c^{(m)}$, which converges to $\alpha_c$. However, by \cite[Theorem~1.4 a)]{loop-perc-Zd}, $\alpha_{\#}(\kappa)=0$ for $d=3$ or $4$ and $\kappa=0$. Thus, $\lim\limits_{m\rightarrow\infty}\alpha_{\#}^{(m)}\neq\alpha_{\#}$.

Next, we present several corollaries of Theorem~\ref{thm: truncated loop percolation}. The following one is based on a comparison of intensity measures.
\begin{corollary}\label{cor: continuity of critical curve}
 For $d\geq 3$, $\alpha\rightarrow\kappa_c(\alpha)$ is a strictly increasing continuous function on $[\alpha_c,\infty[$ and $\kappa_c(\alpha)=0$ for $\alpha\in]0,\alpha_c]$. Consequently, $\alpha_c=\inf\{\alpha\geq 0:\kappa_c(\alpha)>0\}$.
\end{corollary}

For $\kappa=0$, in the sub-critical regime, we have at most polynomial decay of one-arm connectivity by considering one loop connection, see \cite[Theorem~1.2]{loop-perc-Zd}. In contrast, for $\kappa=0$ and $\alpha>\alpha_c$, we have exponential decay for the diameter of a finite cluster. Thus, this loop percolation model provides an example in which we have different behavior of finite clusters in sub-critical and super-critical regime.
\begin{corollary}\label{cor: exponential tail}
 For $d\geq 3$, $\kappa\geq 0$ and $\alpha>\alpha_c(\kappa)$, there exist constants $c=c(d,\alpha,\kappa)>0$, $C=C(d,\alpha,\kappa)<\infty$ and $L_0=L_0(d,\alpha,\kappa)<\infty$ such that for $L\geq L_0$ and $n\geq 1$,
 \begin{equation*}
  \mathbb{P}[\#\mathcal{C}_{\alpha,\kappa}^{\leq L}(0)<\infty,\mathcal{C}_{\alpha,\kappa}^{\leq L}(0)\cap\partial B(n)\neq\emptyset]\leq C(d,\alpha,\kappa)e^{-c(d,\alpha,\kappa)n},
 \end{equation*}
 where $\mathcal{C}_{\alpha,\kappa}^{\leq L}=\{x\in\mathbb{Z}^d:0\overset{\mathcal{L}_{\alpha,\kappa}^{\leq L}}{\longleftrightarrow}x\}$.
\end{corollary}
We give a remark as an immediate consequence of Corollary~\ref{cor: exponential tail}:
\begin{remark}\label{rem: theta L to theta}
 Under the same assumption and with the same constants in Corollary~\ref{cor: exponential tail}, we have
 \[\sup\limits_{L\geq L_0}|\theta^{(L)}_n(\alpha,\kappa)-\theta^{(L)}(\alpha,\kappa)|\leq Ce^{-cn},\]
 where $\theta^{(L)}_n(\alpha,\kappa)=\mathbb{P}[\mathcal{C}_{\alpha,\kappa}^{\leq L}(0)\cap \partial B(n)\neq\emptyset]$ and $\theta^{(L)}(\alpha,\kappa)=\mathbb{P}[\#\mathcal{C}_{\alpha,\kappa}^{\leq L}(0)=\infty]$. Hence, we see that $\lim\limits_{L\rightarrow\infty}\theta^{(L)}(\alpha,\kappa)=\theta(\alpha,\kappa)$.
\end{remark}
The arguments from Theorem~\ref{thm: truncated loop percolation} ii) to Corollary~\ref{cor: exponential tail} is almost the same as Bernoulli bond percolation, see respectively \cite[Theorem~7.2]{GrimmettMR1707339} for Bernoulli bond percolation. We use the independence between disjoint loop soups and the FKG inequality for loops (see \cite[Section 2.1]{LeJanLemaireMR3263044}).

In \cite[Definition 4.1]{Sapozhnikov2014}, A. Sapozhnikov introduced the notions of regular and very regular balls. By \cite[Claim 4.2]{Sapozhnikov2014}, (very) regular ball is always (very) good. The notion of (very) good balls is introduced by M. Barlow in \cite[Definition 1.7]{BarlowMR2094438}, which is, roughly speaking, an assumption of weak Poincar\'{e} inequalities at large enough scales. This assumption implies certain Gaussian estimates and Harnack inequalities of the random walks on the graph and etc, see \cite[Theorem~5.7(a), Theorem~5.11]{BarlowMR2094438}, \cite[Theorem~2.2, Theorem~3.11]{BarlowHamblyMR2471657} and \cite[Theorem~4]{BenjaminiDuminil-CopinKozmaYadin2014}. For the convenience of the reader, we formulate the notion of (very) regular balls for unweighted graphs in our setting.

\begin{defn}\cite[Definition 4.1]{Sapozhnikov2014}\label{defn: very regular balls}
 Let $C_V$, $C_P$ and $C_W\geq 1$ be fixed constants. Let $G$ be an unweighted graph with the vertex set $V(G)$. For an integer $r\geq 1$ and $x\in V(G)$, we say that a ball $B_G(x,r)$ in $G$ centered at $x$ of radius $r$ is $(C_V, C_P, C_W)$-regular if $\# B_G(x,r)\geq C_Vr^d$ and there
exists a set $\mathcal{C}_{B_G(x,r)}$ such that $B_{G}(x,r)\subset\mathcal{C}_{B_G(x,r)}\subset B_{G}(x,C_W r)$ and for any $A\subset \mathcal{C}_{B_G(x,r)}$ with $\#A\leq\frac{1}{2}\cdot\#\mathcal{C}_{B_G(x,r)}$,
\[\#\partial_{\mathcal{C}_{B_G(x,r)}}A\geq \frac{1}{r\sqrt{C_P}}\cdot \#A.\]
We say that $B_G(x,R)$ is $(C_V, C_P, C_W)$-very regular if there exists $N_{B_G(x,R)}\leq R^{\frac{1}{d+2}}$ such that $B_G(y,r)$ is $(C_V, C_P, C_W)$-regular whenever $B_{G}(y,r)\subset B_{G}(x,R)$ and $N_{B_G(x,R)}\leq r\leq R$.
\end{defn}

We prove that the big balls inside the infinite open cluster are very regular with high probabilities:
\begin{theorem}\label{thm: good balls}
 Let $R\geq 1$ be an integer. There exists some $\epsilon\in]0,\frac{1}{d+2}[$. There exist $C_V$, $C_P=C_P(d,\alpha,L_0)<\infty$, $C_W=C_W(L_0)<\infty$, $c=c(d,\alpha,\epsilon)$ and $C=C(d,\alpha,\epsilon)$, such that for all $R\geq 1$,
 \begin{equation}\label{eq: tgb1}
 \mathbb{P}\left[\left.\begin{array}{l}
                        B_{\mathcal{S}_{\infty}}(0,R)\text{ is }(C_V, C_P, C_W)\text{-very regular}\\
                        \text{with }N_{B_{\mathcal{S}_{\infty}}(0,R)}\leq R^{\epsilon}.
                       \end{array}
\right|0\in\mathcal{S}_{\infty}\right]\geq 1-C\exp\{-\exp\{c\sqrt{\log R}\}\},
 \end{equation}
 where $\mathcal{S}_{\infty}$ is the unique infinite open cluster and for $x\in\mathcal{S}_{\infty}$, we denote by $B_{\mathcal{S}_{\infty}}(x,R)$ the ball inside $\mathcal{S}_{\infty}$ of the center $x$ and the radius $R$.
\end{theorem}
\begin{remark}
 We remark that \eqref{eq: tgb1} implies that big enough balls inside the infinite open clusters are very regular by Borel-Cantelli lemma.
\end{remark}

We would like to emphasize that our results are valid for the whole supercritical regime. However, for many models with long range correlations, such properties are demonstrated under additional assumptions of probabilities of certain events inside annuli (which are believed to be true in the whole supercritical regime). For instance, vacant sets of random interlacement is such a long-range percolation model introduced by A.-S. Sznitman, see \cite{SznitmanMR2680403},\cite{TeixeiraMR2824866}, \cite{DrewitzRathSapozhnikovMR3269990} and \cite{DrewitzRathSapozhnikovMR}.

Finally, let's compare with the result of long range bond percolation in \cite{MeesterSteifMR1405960}, where the continuity of critical values for truncated models is proved under the assumption of ``exponential decay'' of one edge connection probability $\mathbb{P}[\text{edge }\{0,x\}\text{ is open}]$. However, the loop percolation model is not a bond percolation model and we have the continuity of critical values of a sequence of truncated loop models with polynomial decays of one loop connection probabilities $\mathbb{P}[\exists \ell\in\mathcal{L}_{\alpha}:0,x\in\ell]$ when $\kappa=0$.

\emph{Organization of the paper}: We fix some notation in Section 2. Then we prove Theorem~\ref{thm: truncated loop percolation} i), Theorem~\ref{thm: truncated loop percolation} ii), Corollary~\ref{cor: continuity of critical curve}, Corollary~\ref{cor: exponential tail} and Theorem~\ref{thm: good balls} in separate sections.

\section{Definition and notation}
We fix several notation:
\begin{itemize}
 \item $B(r)=\{-\lfloor r\rfloor,\ldots,\lfloor r \rfloor\}^d$ and $B(x,r)=x+B(r)$ for $x\in\mathbb{Z}^d$ and $r\geq 0$.
 \item $\partial B(r)=\{x\in\mathbb{Z}^d:|x|_{\infty}=\lfloor r\rfloor \}$ and $\partial B(x,r)=x+\partial B(r)$ for $x\in\mathbb{Z}^d$ and $r\geq 0$.
 \item For $m\geq 1$ and $\vec{k}=(k_1,\ldots,k_d)\in\mathbb{Z}^d$, we define
 \begin{equation}\label{eq: defn box}
  B^{(m)}(\vec{k})\overset{\text{def}}{=}\{k_1m,k_1m+1,\ldots,(k_1+1)m-1\}\times\cdots\times\{k_dm,k_dm+1,\ldots,(k_d+1)m-1\}.
 \end{equation}
 \item $F(n)=\{x\in\partial B(n):x_1=n\}$.
 \item $T(n)=\{x\in\partial B(n):x_1=n,x_j\geq 0\text{ for }j\geq 2\}$.
 \item $T(m,n)=\bigcup\limits_{j=0}^{2m}\{je_1+T(n)\}$, where $e_1=(1,0,\ldots,0)\in\mathbb{Z}^d$.
 \item For a subgraph $G$ and a vertex $A$ of the vertex set $G$, we define the (inner vertex) boundary
 \[\partial_{G}A=\{x\in A:\exists y\in G\setminus A,\{x,y\}\text{ is an edge in }G\}.\]
 \item For $x\in\mathbb{Z}^d$ and a loop $\ell$, we write $x\in\ell$ if $\ell$ covers the vertex $x$.
 \item Define $\Diam(\ell)\overset{\text{def}}{=}\max\limits_{x,y\in\ell}|x-y|_{\infty}$. 
 \item For a vertex set $A$ and a loop $\ell$, we write $\ell\cap A\neq\emptyset$ if $\exists x\in A$ such that $x\in\ell$.
 \item For two vertex sets $A$ and $B$, we write $A\overset{\ell}{\longleftrightarrow}B$ if $\ell\cap A\neq\emptyset$ and $\ell\cap B\neq\emptyset$.
 \item For a vertex set $A$ and a loop $\ell$, we write $\ell\subset A$ if all the vertices of $\ell$ are contained in $A$.
 \item For a non-negative function $\beta$ on the space of loops, let $\mathcal{L}_{\beta}$ be the loop soup of the intensity measure $\nu(\mathrm{d}\ell)=\beta(\ell)\mu_{\kappa=0}(\mathrm{d}\ell)$. We call the function $\beta$ the \emph{intensity function} of the loop soup.\footnote{Note that $\mathcal{L}_{\alpha,\kappa}$ is a special case as its intensity measure $\alpha\mu_{\kappa}(\mathrm{d}\ell)$ equals $\alpha\left(\frac{1}{1+\kappa}\right)^{-|\ell|}\mu_0(\mathrm{d}\ell)$.}
 \item Let $\mathcal{L}$ be a loop soup. For two vertices $x$ and $y$, we write $x\overset{\mathcal{L}}{\longleftrightarrow} y$ if $x$ and $y$ are in the same open cluster. For two vertex sets $A$ and $B$, the notation $A\overset{\mathcal{L}}{\longleftrightarrow} B$ means that there exist $x\in A$ and $y\in B$ such that $x\overset{\mathcal{L}}{\longleftrightarrow} y$. For a vertex $x$ and a set of vertices $B$, we write $x\overset{\mathcal{L}}{\longleftrightarrow} B$ instead of $\{x\}\overset{\mathcal{L}}{\longleftrightarrow} B$. We use the notation $x\overset{\mathcal{L}}{\longleftrightarrow} \infty$ for the event that there exists an infinite open cluster containing the vertex $x$.
 \item For a subgraph $G$ of $\mathbb{Z}^d$, a loop soup $\mathcal{L}$ and two vertices $x$ and $y$, we write $x\overset{G,\mathcal{L}}{\longleftrightarrow} y$ if $x$ and $y$ are connected by open edges inside $G$ given by $\mathcal{L}$. (Recall that an edge is open with respect to $\mathcal{L}$ if the edge is crossed by some loop inside $\mathcal{L}$.) For two vertex sets $A$ and $B$, the notation $A\overset{G,\mathcal{L}}{\longleftrightarrow} B$ means that there exist $x\in A$ and $y\in B$ such that $x\overset{G,\mathcal{L}}{\longleftrightarrow} y$. For a vertex $x$ and a set of vertices $B$, we write $x\overset{G,\mathcal{L}}{\longleftrightarrow} B$ instead of $\{x\}\overset{G,\mathcal{L}}{\longleftrightarrow} B$.
 \item Let $\beta$ be an intensity function. For a vertex $x$, $\mathcal{C}_{\beta}(x)=\{y:y\overset{\mathcal{L}_{\beta}}{\longleftrightarrow}x\}$. For a vertex set $A$, $\mathcal{C}_{\beta}(A)=\bigcup\limits_{x\in A}\mathcal{C}_{\beta}(x)$.
 \item $\theta_n(\beta)=\mathbb{P}[0\overset{\mathcal{L}_{\beta}}{\longleftrightarrow}\partial B(n)]$ and $\theta(\beta)=\mathbb{P}[0\overset{\mathcal{L}_{\beta}}{\longleftrightarrow}\infty]$. When $\beta(\mathrm{d}\ell)=\alpha\left(\frac{1}{1+\kappa}\right)^{-|\ell|}\mu_0(\mathrm{d}\ell)$, we write $\theta_n(\alpha,\kappa)$ and $\theta(\alpha,\kappa)$. (And we write $\theta_n(\alpha)$ and $\theta(\alpha)$ when $\kappa=0$.)
 \item For a subset $K$ of vertices and a loop soup $\mathcal{L}$, we define
\[(\mathcal{L})_K\overset{\mathrm{def}}{=}\{\ell\in\mathcal{L}:\ell\cap K\neq\emptyset\}\text{ and }(\mathcal{L})^K\overset{\mathrm{def}}{=}\{\ell\in\mathcal{L}:\ell\subset K\}.\]
\end{itemize}

\section{Proof of Theorem~\ref{thm: truncated loop percolation} i)}
\emph{We assume $\kappa=0$ since the argument for $\kappa>0$ is similar.}
For Bernoulli bond percolation, if we add additional edges on the diagonals of cubes, then the critical value strictly decreases, see \cite[Section 3.2]{GrimmettMR1707339} for a precise statement. In this section, we adapt the argument for our loop percolation. The key is a comparison of partial derivatives similar to \cite[Lemma~3.5]{GrimmettMR1707339}:
\begin{lemma}\label{lem: derivative compare}
 Let $\beta:\{\text{loops}\}\rightarrow[0,\infty[$ be an intensity function such that $\beta(\ell)=\beta_i$ for all $\ell$ such that $|\ell|=2i$. Suppose that either $(\beta_i)_i$ is finally zero or $\sum\limits_{0\in\ell}\mu(\ell)<\infty$. Then, for all $i<j$, there exists a finite constant $C=C(i,j)$ such that
 \begin{equation}\label{eq: ldc1}
  \frac{\partial}{\partial\beta_i}\theta_n(\beta)\leq C\frac{\partial}{\partial \beta_j}\theta_n(\beta).
 \end{equation}
\end{lemma}
\begin{remark}
 For all $i>j$ and $a>0$, there exists $C=C(i,j,a)<\infty$, non-increasing in $a$, such that Equation \eqref{eq: ldc1} holds as long as $\beta_j\geq a$. Since we don't need this result, we omit this part.
\end{remark}

Firstly, let's explain how Lemma~\ref{lem: derivative compare} implies Theorem~\ref{thm: truncated loop percolation} i):
\begin{proof}[Proof of Theorem~\ref{thm: truncated loop percolation} i)]
 Take $\epsilon=\left(\sum\limits_{i=1}^{m}C(i,m+1)\right)^{-1}$, where $(C(i,m+1))_i$ is the same as in Lemma~\ref{lem: derivative compare}. For $t\in[0,1]$, set $\beta_{m+1}(t)=\alpha-\alpha t$, $\beta_i(t)=\alpha+\epsilon\alpha t$ for $i=1,\ldots,m$ and $\beta_{j}(t)=0$ for $j\geq m+2$. Note that $\frac{\partial\theta_n}{\partial \beta_i}\geq 0$ for all $i$. Hence, by Lemma~\ref{lem: derivative compare},
 \[\frac{\mathrm{d}\theta_n(\beta(t))}{\mathrm{d}t}=-\alpha\frac{\partial\theta_n}{\partial\beta_{m+1}}+\epsilon\alpha\sum\limits_{i=1}^{m}\frac{\partial\theta_n}{\partial\beta_{i}}\leq 0.\]
 Thus, $\theta_n(\beta(0))\geq \theta_n(\beta(1))$. By taking $n\rightarrow\infty$, we see that $\theta(\beta(0))\geq \theta(\beta(1))$, which implies that $(1+\epsilon)\alpha_c^{(2m+2)}\leq \alpha_c^{(2m)}$. The proof follows from the fact that $\alpha_c^{(2m)}>0$ for $m\geq 1$.
\end{proof}

Next, to prove Lemma~\ref{lem: derivative compare}, we need a version of Russo's formula. Its proof is basically the same as in Bernoulli bond percolation (see \cite[Theorem~7.2]{GrimmettMR1707339}) and we leave it to the reader. A more general Russo's formula for Poisson point process was given by S. A. Zuev \cite{ZuevMR1220977}.
\begin{lemma}[Russo's formula]\label{lem: russo formula}
 Let $\beta:\{loops\}\rightarrow[0,\infty[$ be an intensity function. For a loop $\ell$,
 \[\frac{\partial\mathbb{P}[0\overset{\mathcal{L}_{\beta}}{\longleftrightarrow}\partial B(n)]}{\partial\beta(\ell)}
  =\mu(\ell)\mathbb{P}\left[\mathcal{C}_{\beta}(0)\overset{\ell}{\longleftrightarrow}\mathcal{C}_{\beta}(\partial B(n)),\mathcal{C}_{\beta}(0)\cap\partial B(n)=\emptyset\right].\]
\end{lemma}
Finally, we end this section by proving Lemma~\ref{lem: derivative compare}.
\begin{proof}[Proof of Lemma~\ref{lem: derivative compare}]
 We fix $i<j$. By Russo's formula,
 \begin{equation}\label{eq: ldc2}
  \frac{\partial \theta_n(\beta)}{\partial \beta_i}=\sum\limits_{\ell:|\ell|=2i}\mu(\ell)\mathbb{P}\left[\mathcal{C}_{\beta}(0)\overset{\ell}{\longleftrightarrow}\mathcal{C}_{\beta}(\partial B(n)),\mathcal{C}_{\beta}(0)\cap\partial B(n)=\emptyset\right].
 \end{equation}
 For $i,j\geq 1$, there exists $C'(i,j)<\infty$ such that
 \begin{equation}\label{eq: ldc3}
  \sup\limits_{\eta:|\eta|=2j}\#\{\ell:|\ell|=2i,\eta\cap\ell\neq\emptyset\}<C'(i,j),
 \end{equation}
 \begin{equation}\label{eq: ldc4}
  \sup\limits_{\ell:|\ell|=2i}\mu(\ell)\leq C'(i,j)\inf\limits_{\eta:|\eta|=2j}\mu(\eta).
 \end{equation}
 Then, by Russo's formula \eqref{eq: ldc2} with Inequalities \eqref{eq: ldc3} and \eqref{eq: ldc4}, \eqref{eq: ldc1} will be deduced from the following statement: for all fixed loops $\ell$ such that $|\ell|=2i$ and $\ell\cap B(n)\neq\emptyset$,
 \begin{multline}\label{eq: ldc5}
  \mathbb{P}\left[\mathcal{C}_{\beta}(0)\overset{\ell}{\longleftrightarrow}\mathcal{C}_{\beta}(\partial B(n)),\mathcal{C}_{\beta}(0)\cap\partial B(n)=\emptyset\right]\\
  \leq \sum\limits_{\eta:|\eta|=2j}\mathbb{P}\left[\mathcal{C}_{\beta}(0)\overset{\eta}{\longleftrightarrow}\mathcal{C}_{\beta}(\partial B(n)),\mathcal{C}_{\beta}(0)\cap\partial B(n)=\emptyset\right].
 \end{multline}
 Indeed, the constant $C(i,j)$ in \eqref{eq: ldc1} is chosen to be $(C'(i,j))^2$. Note that for $i<j$ and $\ell$ such that $|\ell|=2i$, there exists at least one loop $\eta$ such that $|\eta|=2j$ and that $\eta$ covers the same set of vertices as $\ell$. Hence, \eqref{eq: ldc5} follows.
\end{proof}
\section{Proof of Theorem~\ref{thm: truncated loop percolation} ii)}
\emph{We assume $\kappa=0$ since the argument for $\kappa>0$ is similar and even simpler.}

We will use the same dynamic renormalization schema as in \cite[Theorem~7.2]{GrimmettMR1707339} for Bernoulli bond percolation in slabs. The key of the proof of \cite[Theorem~7.2]{GrimmettMR1707339} is a sprinkling lemma (\cite[Lemma~7.17]{GrimmettMR1707339}), which states that the disconnection can be turned into connection by a local small increase in the percolation parameter $p$. We get an analogue for loop percolation. It is crucial that the modification of the intensity function is local in the sprinkling lemma (Lemma~\ref{lem: sprinkling}).
\begin{defn}[Seed event]
 Let $\beta:\{\text{\text{loops}}\}\rightarrow\mathbb{R}_{+}$ be the intensity function of the loop soup. Let $m\geq 1$, we call a box $B(x,m)$ a $\beta$-\emph{seed} if every edge in $B(x,m)$ is covered by some loop in $(\mathcal{L}_{\beta})^{B(x,m)}$. We set
 \begin{equation*}
  K(m,n,\beta)=\{\text{union of the }\beta\text{-seeds lying within }T(m,n)\}.
 \end{equation*}
 When $\beta\equiv\alpha$ is a constant function, we write $\alpha$-seed and $K(m,n,\alpha)$.
\end{defn}
\begin{lemma}\label{lem: sprinkling}
 Suppose $d\geq 3$ and $\alpha>\alpha_c$. For $\epsilon,\delta>0$ and an intensity function $\gamma:\{\text{loops}\}\rightarrow[\alpha+\delta,A]$, there exist integers $m=m(d,\alpha,A,\epsilon,\delta)$ and $n=n(d,\alpha,A,\epsilon,\delta)$ such that $2m<n$ and the following property holds. Let $R$ be such that $B(m)\subset R\subset B(n)$. Define $\gamma':\{\text{loops}\}\rightarrow\mathbb{R}_{+}$ as follows: for a loop $\ell$,
 \[\gamma'(\ell)=\left\{\begin{array}{ll}
 \gamma(\ell)+\delta & \text{ if }\ell\subset B(2n-1),\ell\cap \partial B(n-1)\neq\emptyset,\ell\cap R\neq \emptyset,\\
 \gamma(\ell) & \text{ otherwise.}
  \end{array}\right.\]
 Define events
 \begin{equation*}
  G=\left\{R\overset{(\mathcal{L}_{\gamma'})_{B(n-1)}^{B(2n-1)}}{\longleftrightarrow}K(m,n,\gamma')\right\}\text{ and }H=\left\{R\overset{(\mathcal{L}_{\gamma})_{B(n-1)}^{B(2n-1)}}{\centernot\longleftrightarrow}K(m,n,\gamma)\right\}.
 \end{equation*}
 Then, $\mathbb{P}[G|H]>1-\epsilon$.
\end{lemma}
\begin{proof}[Proof of Theorem~\ref{thm: truncated loop percolation} ii)]
 
\end{proof}

The argument from Lemma~\ref{lem: sprinkling} to Theorem~\ref{thm: truncated loop percolation} ii) is the same as the renormalization argument in Bernoulli bond percolation case. It is based on a comparison with dependent site percolation. We omit this part and refer to Pages 154-162 in \cite{GrimmettMR1707339} for the details. We prepare several lemmas in steps and then prove Lemma~\ref{lem: sprinkling} in the end of this section.

We first state a result of the loop measure $\mu$ on the one loop connection.
\begin{lemma}(\cite[Lemma~2.7 a)]{loop-perc-Zd}+\cite[Proposition 6.5.1]{LawlerLimicMR2677157})\label{lem: mu annulus}
 For $d\geq 3$ and $\lambda>1$, there exists a positive constant $C=C(d,\lambda)<\infty$ such that for all $N\geq 1$ and $M\geq \lambda N$,
 \begin{equation*}
  \mu(\ell\cap B(N)\neq\emptyset, \ell\cap\partial B(M)\neq\emptyset)\leq C\cdot (N/M)^{d-2}.
 \end{equation*}
\end{lemma}

The following lemma states that if there is a connection from the inner box $B(m)$ to the outer box $\partial B(\sqrt{n})$, then the contribution of the loops crossing $\partial B(n)$ is negligible.
\begin{lemma}\label{lem: truncated crossing}
 For fixed $m\geq 1$, 
 \begin{equation*}
  \lim\limits_{n\rightarrow\infty}\mathbb{P}[B(m)\overset{\mathcal{L}_{\alpha}}{\longleftrightarrow}\partial B(\sqrt{n}),B(m)\overset{(\mathcal{L}_{\alpha})^{B(n-1)}}{\centernot\longleftrightarrow}\partial B(\sqrt{n})]=0.
 \end{equation*}
\end{lemma}
\begin{proof}
 Since the event $\left\{B(m)\overset{\mathcal{L}_{\alpha}}{\longleftrightarrow}\partial B(\sqrt{n}),B(m)\overset{(\mathcal{L}_{\alpha})^{B(n-1)}}{\centernot\longleftrightarrow}\partial B(\sqrt{n})\right\}$ implies that there exists  $\ell\in\mathcal{L}_{\alpha}$ such that $\partial B(\sqrt{n})\overset{\ell}{\longleftrightarrow}\partial B(n)$, we have that
 \begin{multline*}
  \mathbb{P}\left[B(m)\overset{\mathcal{L}_{\alpha}}{\longleftrightarrow}\partial B(\sqrt{n}),B(m)\overset{(\mathcal{L}_{\alpha})^{B(n-1)}}{\centernot\longleftrightarrow}\partial B(\sqrt{n})\right]\\
  \leq 1-\exp\{-\alpha\mu(\ell\cap\partial B(n)\neq\emptyset,\ell\cap \partial B(\sqrt{n})\neq\emptyset).
 \end{multline*}
 Then, the result follows from Lemma~\ref{lem: mu annulus}.
\end{proof}
The next lemma states that if $B(m)\overset{\mathcal{L}_{\alpha}}{\longleftrightarrow}\partial B(\sqrt{n})$ and if we increase the intensity a bit, then necessarily, we have lots of loops contained in $B(2n-1)$, which intersect $\partial B(n)$ and are connected to $B(m)$ through the loops inside $B(2n-1)$.
\begin{lemma}\label{lem: many trancated loops}
 For $\alpha,\delta>0$ and $n\geq 1$, define an intensity function $\beta_n:\{loops\}\rightarrow [0,\infty[$ as follows:
 \begin{equation*}
  \beta_n(\ell)=\left\{\begin{array}{ll}
                      \alpha+\delta & \text{if }\ell\cap\partial B(n)\neq\emptyset,\ell\subset B(2n-1),\\
                      \alpha & \text{otherwise.}
                     \end{array}
  \right.
 \end{equation*}
 Denote by $\mathcal{C}(m,n,\beta_n)$ the cluster of vertices which can be connected to $B(m)$ by loops $(\mathcal{L}_{\beta_n})^{B(n-1)}=(\mathcal{L}_\alpha)^{B(n-1)}$ contained in $B(n-1)$:
 \[\mathcal{C}(m,n,\beta_n)=B(m)\cup\{x\in B(n-1):x\overset{(\mathcal{L}_{\beta_n})^{B(n-1)}}{\longleftrightarrow}B(m)\}.\]
 Denote by $\mathcal{O}(m,n,\beta_n)$ the sub-multiset of loops $(\mathcal{L}_{\beta_n})^{B(2n-1)}$ which is contained in $B(2n-1)$ and intersects both $\partial B(n)$ and $\mathcal{C}(m,n,\beta_n)$. Then, for fixed $m,k\geq 1$ and $\alpha,\delta>0$,
 \begin{equation*}
  \lim\limits_{n\rightarrow\infty}\mathbb{P}[B(m)\overset{\mathcal{L}_{\alpha}}\longleftrightarrow \partial B(\sqrt{n}),\#\mathcal{O}(m,n,\beta_n)\leq k]=0.
 \end{equation*}
\end{lemma}
\begin{proof}
 By Lemma~\ref{lem: truncated crossing}, it is enough to show that
 \[\lim\limits_{n\rightarrow\infty}\mathbb{P}[B(m)\overset{(\mathcal{L}_{\alpha})^{B(n-1)}}\longleftrightarrow \partial B(\sqrt{n}),\#\mathcal{O}(m,n,\beta_n)\leq k]=0.\]
 Consider $\widetilde{\mathcal{O}}(m,n,\beta_n)$ the subset of loops $\mathcal{L}_{\beta_n}$ which intersect both $\partial B(2n)$ and $\mathcal{C}(m,n,\beta_n)$.
 Conditionally on the loops $(\mathcal{L}_{\alpha})^{B(n-1)}=(\mathcal{L}_{\beta_n})^{B(n-1)}$ inside $B(n-1)$, $\#\mathcal{O}(m,n,\beta_n)$ follows a Poisson distribution of the expectation \[(\alpha+\delta)\mu(\ell\cap \partial B(n)\neq\emptyset,\ell\cap\mathcal{C}(m,n,\beta_n)\neq\emptyset,\ell\subset B(2n-1)),\]
 and $\#\widetilde{\mathcal{O}}(m,n,\beta_n)$ is another independent Poisson random variable of the expectation
 \[\alpha\mu(\ell\cap \partial B(2n)\neq\emptyset,\ell\cap\mathcal{C}(m,n,\beta_n)\neq\emptyset).\] By Lemma~\ref{lem: mu annulus}, there exists a universal $C=C(d)<\infty$ such that
 \[\alpha\mu(\ell\cap \partial B(2n)\neq\emptyset,\ell\cap\mathcal{C}(m,n,\beta_n)\neq\emptyset)\leq \alpha\mu(\ell\cap \partial B(2n)\neq\emptyset,\ell\cap B(n-1)\neq\emptyset)\leq C(d)\alpha.\]
 Accordingly,
 \begin{multline}\label{eq: mtl2}
  \mathbb{P}[B(m)\overset{(\mathcal{L}_{\alpha})^{B(n-1)}}\longleftrightarrow \partial B(\sqrt{n}),\#\mathcal{O}(m,n,\beta_n)\leq k]\\
  \leq e^{\alpha C(d)}\mathbb{P}[B(m)\overset{(\mathcal{L}_{\alpha})^{B(n-1)}}\longleftrightarrow \partial B(\sqrt{n}),\#\mathcal{O}(m,n,\beta_n)\leq k,\#\widetilde{\mathcal{O}}(m,n,\beta_n)=0].
 \end{multline}
 Conditionally on $(\mathcal{L}_{\alpha})^{B(n-1)}$, by using the natural monotone coupling between $\mathcal{L}_{\alpha}$ and $\mathcal{L}_{\beta_n}$, we get that
 \[\mathbb{P}[\#\mathcal{O}(m,n,\beta_n)\leq k|(\mathcal{L}_{\alpha})^{B(n-1)}]\leq \left(\frac{\alpha+\delta}{\delta}\right)^{k}\mathbb{P}[\#\mathcal{O}(m,n,\alpha)=0|(\mathcal{L}_{\alpha})^{B(n-1)}].\]
 Then, by noting that $\lim\limits_{n\rightarrow\infty}\mathbb{P}[B(m)\overset{\mathcal{L}_{\alpha}}{\longleftrightarrow} \partial B(\sqrt{n})]=\lim\limits_{n\rightarrow\infty}\mathbb{P}[B(m)\overset{\mathcal{L}_{\alpha}}{\longleftrightarrow} \partial B(n)]=\mathbb{P}[B(m)\overset{\mathcal{L}_{\alpha}}{\longleftrightarrow}\infty]$,
 \begin{align*}
  \eqref{eq: mtl2}\leq &e^{\alpha C(d)}\left(\frac{\alpha+\delta}{\delta}\right)^{k}\mathbb{P}[B(m)\overset{\mathcal{L}_{\alpha}}{\longleftrightarrow} \partial B(\sqrt{n}),B(m)\overset{\mathcal{L}_{\alpha}}{\centernot\longleftrightarrow}\partial B(n)]\\
  =&e^{\alpha C(d)}\left(\frac{\alpha+\delta}{\delta}\right)^{k}\left(\mathbb{P}[B(m)\overset{\mathcal{L}_{\alpha}}{\longleftrightarrow} \partial B(\sqrt{n})]-\mathbb{P}[B(m)\overset{\mathcal{L}_{\alpha}}{\centernot\longleftrightarrow}\partial B(n)]\right)\overset{n\rightarrow\infty}{\longrightarrow}0.\qedhere
 \end{align*}
\end{proof}

The following lemma is an analogue of Equation (7.16) in the proof of \cite[Lemma~7.9]{GrimmettMR1707339}.
\begin{lemma}\label{lem: many truncated connections to the outer box}
 Let $\alpha,\delta,\beta_n$ be the same as in Lemma~\ref{lem: many trancated loops} and let
 \[U(m,n,\beta_n)=\{x\in \partial B(n):x\overset{(\mathcal{L}_{\beta_n})^{B(2n-1)}_{B(n-1)}}{\longleftrightarrow}B(m)\}.\]
Then, for all fixed $k\geq 1$,
 \begin{equation}\label{eq: mtcttob0}
 \lim\limits_{n\rightarrow\infty}\mathbb{P}[\#U(m,n,\beta_n)\leq k,B(m)\overset{(\mathcal{L}_{\alpha})^{B(n-1)}}{\longleftrightarrow} \partial B(\sqrt{n})]=0.
 \end{equation}
\end{lemma}

\begin{proof}
 Denote by $\mathcal{E}(m,n,\beta_n)$ the point measure of excursions outside of $B(n-1)$ of the loops in $\mathcal{O}(m,n,\beta_n)$, where $\mathcal{O}(m,n,\beta_n)$ is defined in Lemma~\ref{lem: many trancated loops}. Let $\mathcal{E}(m,n,\beta_n)(1)$ be the total mass of $\mathcal{E}(m,n,\beta_n)$, i.e. the number of excursions with multiplicity. Then, as a consequence of Lemma~\ref{lem: many trancated loops}, for any fixed $q\geq 1$,
\begin{equation}\label{eq: mtcttob3}
 \lim\limits_{n\rightarrow\infty}\mathbb{P}[\mathcal{E}(m,n,\beta_n)(1)\leq q,B(m)\overset{(\mathcal{L}_{\beta_n})^{B(n-1)}}{\longleftrightarrow} \partial B(\sqrt{n})]=0.
\end{equation}
Conditionally on the position of the start and end points of an excursion, the excursion follows the normalized distribution of the excursion outside of $B(n-1)$ with that given pair of start and end points. Note that we have independence between excursions. Moreover, with probability uniformly bounded from below by $p(d,k)>0$, the excursion covers at least $k$ vertices on $\partial B(n)$. Hence, \eqref{eq: mtcttob0} follows from \eqref{eq: mtcttob3}.
\end{proof}

We will deduce the following key lemma from Lemma~\ref{lem: many truncated connections to the outer box}. The argument is an adaptation from the case of Bernoulli bond percolation, see e.g. \cite[Lemma~7.9]{GrimmettMR1707339}.
\begin{lemma}\label{lem: a connection to a seed}
 If $\theta(\alpha)>0$, for all $\delta,\eta>0$, there exist $m=m(d,\alpha,\eta)$ and $n=n(d,\alpha,\delta,\eta)>2m$ such that
 \begin{equation*}
  \mathbb{P}[B(m)\overset{(\mathcal{L}_{\beta_n})_{B(n-1)}^{B(2n-1)}}{\longleftrightarrow} K(m,n,\alpha)]>1-\eta,
 \end{equation*}
 where $\beta_n$ is the same as in Lemma~\ref{lem: many trancated loops} and~\ref{lem: many truncated connections to the outer box}.
\end{lemma}
\begin{proof}
 Since $\theta(\alpha)>0$, we pick $m=m(d,\alpha,\eta)$ such that
 \[\mathbb{P}[B(m)\overset{\mathcal{L}_{\alpha}}{\longleftrightarrow}\infty]>1-(\eta/4)^{d\cdot 2^d}.\]
 Pick $M$ such that $pr\overset{\mathrm{def}}{=}\mathbb{P}[B(m)\text{ is a }\alpha\text{-seed}]>1-(\eta/2)^{1/M}$. We assume $2m+1$ divides $n+1$. Let $V(m,n,\beta_n)=\{x\in T(n):x\overset{(\mathcal{L}_{\beta_n})_{B(n-1)}^{B(2n-1)}}{\longleftrightarrow}B(m)\}$. If $\#V(m,n,\beta_n)\geq (2m+1)^{d-1}M$, then $B(m)$ is joined by loops $(\mathcal{L}_{\beta_n})_{B(n-1)}^{B(2n-1)}$ to at least $M$ of these squares outside of $B(n)$. Therefore, by independence between the loops $(\mathcal{L}_{\beta_n})_{B(n-1)}$ intersecting the box $B(n-1)$ and the loops $(\mathcal{L}_{\beta_n})^{(B(n-1))^c}$ avoiding $B(n-1)$, we have that
 \begin{align}\label{eq: actas1}
  \mathbb{P}[B(m)\overset{(\mathcal{L}_{\beta_n})_{B(n-1)}}{\longleftrightarrow} K(m,n,\alpha)]&\geq (1-(1-pr)^M)\mathbb{P}[\#V(m,n,\beta_n)\geq (2m+1)^{d-1}M]\notag\\
  &\geq (1-\eta/2)\mathbb{P}[\#V(m,n,\beta_n)\geq (2m+1)^{d-1}M].
 \end{align}
 Since $\partial B(n)$ has $d\cdot 2^d$ copies of $T(n)$, by FKG inequality and symmetry,
 \[(\mathbb{P}[\#V(m,n,\beta_n)<(2m+1)^{d-1}M])^{d\cdot 2^d}\leq\mathbb{P}[\#U(m,n,\beta_n)\leq d\cdot 2^d (2m+1)^{d-1}M].\]
 By applying Lemma~\ref{lem: many truncated connections to the outer box} with $k=d\cdot 2^d (2m+1)^{d-1}M$, we pick $n=n(d,\alpha,\delta,\eta)$ large enough such that
 \[\mathbb{P}[\#U(m,n,\beta_n)\leq d\cdot 2^d (2m+1)^{d-1}M,B(m)\overset{\mathcal{L}_{\alpha}}{\longleftrightarrow}\infty]\leq (\eta/4)^{d\cdot 2^d}.\]
 Then,
 \begin{multline*}
  \mathbb{P}[\#U(m,n,\beta_n)\leq d\cdot 2^d (2m+1)^{d-1}M]\leq\mathbb{P}[B(m)\overset{\mathcal{L}_{\alpha}}{\centernot{\longleftrightarrow}}\infty]\\
  +\mathbb{P}[\#U(m,n,\beta_n)\leq d\cdot 2^d (2m+1)^{d-1}M,B(m)\overset{\mathcal{L}_{\alpha}}{\longleftrightarrow}\infty]\leq 2(\eta/4)^{d\cdot 2^d}.
 \end{multline*}
 Consequently, for the same $n$,
 \begin{equation}\label{eq: actas2}
  \mathbb{P}[\#V(m,n,\beta_n)<(2m+1)^{d-1}M]\leq 2^{2^{-d}/d}\eta/4<\eta/2.
 \end{equation}
 The result follows from \eqref{eq: actas1} and \eqref{eq: actas2}.
\end{proof}
We end this section by the proof of the sprinkling lemma (Lemma~\ref{lem: sprinkling}).
\begin{proof}[Proof of Lemma~\ref{lem: sprinkling}]
 Consider the random subset of $B(2n-1)\setminus R$:
 \[\mathcal{S}_{\gamma}(R,m,n)=\left\{x\notin R:x\overset{(\mathcal{L}_{\gamma})^{B(2n-1)\setminus R}_{B(n-1)}}{\longleftrightarrow}K(m,n,\gamma)\right\}.\]
 By definition of $\gamma'$, we see that $\mathcal{S}_{\gamma'}(R,m,n)=\mathcal{S}_{\gamma}(R,m,n)$. Let
 \begin{align*}
  \mu(R,m,n)=&\sum\limits_{\ell}\gamma(\ell)\mu(\ell\cap\mathcal{S}_{\gamma}(R,m,n)\neq\emptyset,\ell\cap R\neq\emptyset,\ell\subset B(2n-1)),\\
  \mu(R,m,n,\delta)=&\sum\limits_{\ell}\delta\mu(\ell\cap\mathcal{S}_{\gamma}(R,m,n)\neq\emptyset,\ell\cap R\neq\emptyset,\ell\subset B(2n-1)).
 \end{align*}
 Define $\mathcal{O}(R,m,n,\gamma)\overset{\mathrm{def}}{=}\{\ell\in\mathcal{L}_{\gamma}:\ell\cap \mathcal{S}_{\gamma}(R,m,n)\neq\emptyset,\ell\cap R\neq\emptyset,\ell\subset B(2n-1)\}$. Similarly, we define $\mathcal{O}(R,m,n,\gamma')$ by replacing $\mathcal{L}_{\gamma}$ by $\mathcal{L}_{\gamma'}$ in the definition. By definition of loop soup, we have that
 \begin{align*}
  &\mathbb{P}[H]=\mathbb{P}[\mathcal{O}(R,m,n,\gamma)=\emptyset]=\mathbb{E}[\exp\{-\mu(R,m,n)\}],\\
  &1-\mathbb{P}[G|H]=\mathbb{P}[\mathcal{O}(R,m,n,\gamma')=\emptyset|\mathcal{O}(R,m,n,\gamma)=\emptyset]=\mathbb{E}[\exp\{-\mu(R,m,n,\delta)\}].
 \end{align*}
Since $\mu(R,m,n,\delta)\geq \frac{\delta}{|\gamma|_{\infty}}\mu(R,m,n)$, by H\"{o}lder's inequality,
 \[1-\mathbb{P}[G|H]\leq (\mathbb{P}[H])^{\min(\delta/|\gamma|_{\infty},1)}.\]
 For $\epsilon>0$ and $\delta>0$, we choose $\eta$ such that $\eta^{\min(\delta/|\gamma|_{\infty},1)}<\epsilon$ and $m=m(d,\alpha,\eta),n=n(d,\alpha,\delta,\eta)$ as in Lemma~\ref{lem: a connection to a seed}. Since $\gamma\geq \alpha+\delta\geq \beta_n$, $(\mathcal{L}_{\beta_n})_{B(n-1)}^{B(2n-1)}$ is stochastically dominated by $(\mathcal{L}_{\gamma})_{B(n-1)}^{B(2n-1)}$. Therefore, $\mathbb{P}[H]\leq \eta$ and $\mathbb{P}(G|H)>1-\epsilon$.
\end{proof}

\section{Proof of Corollary~\ref{cor: continuity of critical curve}}
The non-decrease of $\alpha\rightarrow\kappa_c(\alpha)$ was proven in \cite[Proposition 4.3]{LeJanLemaireMR3263044}. By a direct comparison of intensity measures, we prove the continuity of $\alpha\rightarrow\kappa_c(\alpha)$ in \cite[Proposition 7.4]{loop-perc-Zd}. By \cite[Theorem~1.1]{loop-perc-Zd}, we see that $\kappa_c(\alpha)=0$ for $\alpha<\alpha_c$. We will prove the increase of $\alpha\rightarrow\kappa_c(\alpha)$ is strict for $\alpha\geq\alpha_c$ by Theorem~\ref{thm: truncated loop percolation}. We give the proof for the strict increase of $\alpha\rightarrow\kappa_c(\alpha)$ at $\alpha_c$. The general case is almost the same and is left to the reader. Note that it is enough to prove that for $\alpha>\alpha_c$, there exists $\kappa_0=\kappa_0(\alpha)>0$ such that $\theta(\alpha,\kappa_0)>0$. Set $\alpha'=(\alpha+\alpha_c)/2$. By Theorem~\ref{thm: truncated loop percolation} ii), we choose $n=n(\alpha)$ large enough such that $\alpha_c^{(n)}<\alpha'$. Next, by a comparison of intensity measure, there exists a small enough $\kappa_0$ such that $\mathcal{L}_{\alpha,\kappa_0}$ stochastically dominates $\mathcal{L}_{\alpha',0}^{\leq n}$. In particular, there exists a percolation for $\mathcal{L}_{\alpha,\kappa_0}$ and the proof is complete.

\section{Proof of Corollary~\ref{cor: exponential tail}}
For $d\geq 3$, the proof is the same as in \cite[Theorem~1]{ChayesChayesNewmanMR905331} with minor modification. One can also find the proof in \cite[Theorem~8.21]{GrimmettMR1707339}. The idea is to use slab percolation. By Theorem~\ref{thm: truncated loop percolation}, for a large enough slab of width at least $L_0(d,\alpha,\kappa)$, there is a percolation by loops inside that slab by $\mathcal{L}_{\alpha,\kappa}^{\leq L_0}$. We assume that $L\geq L_0$. We divide $\mathbb{Z}^d$ into parallel slabs. Conditionally on the past, with probability strictly smaller than some $c=c(\alpha,\kappa)<1$, the cluster $\mathcal{C}_{\alpha,\kappa}^{\leq L}(0)$ can pass one slab without intersecting the infinite cluster formed by loops inside that slab. Thus, the probability in Corollary~\ref{cor: exponential tail} decays exponentially fast. In this Markovian loop percolation model, in place of independence between edges, we use independence between loops inside a slab and loops intersecting the complementary of the slab. This is a minor modification and we leave the details to the reader.

\section{Geometry of the infinite cluster in supercritical regime}
Throughout the section, we assume that $d\geq 3$ and $\alpha>\alpha_c(d)$. Most of the time, we give the proof for $\kappa=0$ as the proof is simpler for $\kappa>0$.

In a percolation model with long range dependences, several geometric properties (\cite{DrewitzRathSapozhnikovMR3269990,SapozhnikovProcacciaRosenthal2013,Sapozhnikov2014}) were obtained for supercritical phase under certain assumptions from \cite{DrewitzRathSapozhnikovMR3269990}:
\begin{itemize}
 \item[\textbf{P1}] Every lattice shift is measure preserving and ergodic.
 \item[\textbf{P2}] The model is monotone in parameter $\alpha$.
 \item[\textbf{P3}] Decoupling inequality.
 \item[\textbf{S1}] Local connectness: for a big box $B(R)$ around $0$, with an overwhelming probability, there exists a big cluster (of diameter at least $R$) intersecting $B(R)$ and two vertices inside $B(R)$ which belong to some big clusters (of diameter at least $R/10$) are connected inside $B(2R)$.
 \item[\textbf{S2}] Continuity of percolation probability: $\mathbb{P}[0\overset{\mathcal{L}_{\alpha}}{\longleftrightarrow}\infty]$ is positive and continuous on $]\alpha_c,\infty[$.
\end{itemize}
We refer to \cite{DrewitzRathSapozhnikovMR3269990} for precise formulations. It is relatively easy to verify conditions \textbf{P1}, \textbf{P2}, \textbf{S1} and \textbf{S2} for the loop percolations. Indeed, for SRW loop soup percolation, condition \textbf{P1} is verified in \cite[Proposition 3.2]{loop-perc-Zd} and condition \textbf{P2} is straightforward from the definition of Poisson point process. For condition \textbf{S1}, we deduce it from Theorem~\ref{thm: truncated loop percolation} and Corollary~\ref{cor: exponential tail} by following the argument in \cite[Lemma~7.89]{GrimmettMR1707339}, see details in Subsection~\ref{sect: S1}. Condition \textbf{S2} also follows from classical arguments. Note that for $d\geq 3$ and $n\geq1 $, $\sum\limits_{\ell:\ell\cap B(n)\neq\phi}\mu(\ell)<\infty$. Hence, by Russo's formula (Lemma~\ref{lem: russo formula}), $\mathbb{P}[0\overset{\mathcal{L}_{\alpha}}{\longleftrightarrow} B(n)]$ is continuous in $\alpha$. Thus, as the decreasing limit of non-decreasing continuous function, $\mathbb{P}[0\overset{\mathcal{L}_{\alpha}}{\longleftrightarrow}\infty]$ is right-continuous in $\alpha$. By similar argument for Bernoulli percolation (see \cite{vandenBergKeaneMR737388} or \cite[Lemma~8.10]{GrimmettMR1707339}), we get the left-continuity for $\alpha>\alpha_c$. (We use the standard coupling for Poisson point process and the uniqueness of the infinite cluster was proven in \cite[Proposition 3.3]{loop-perc-Zd}.) 

However, we are not able to prove the decoupling inequalities \textbf{P3}. Thus, we cannot apply the general results in \cite{DrewitzRathSapozhnikovMR3269990, SapozhnikovProcacciaRosenthal2013,Sapozhnikov2014} to the SRW loop soup percolation model. Instead, we follow closely the proof strategy in \cite{Sapozhnikov2014} for truncated models with bounded range independences and show that the untruncated model can be viewed as small perturbation of truncated models. Some care is needed, e.g. the isoperimetric inequalities are not a monotone property. As pointed out by A. Sapozhnikov, it suffices to have a good control of the volume increase of the infinite cluster within a box after we add the big loops $\mathcal{L}_{\alpha}^{>L}$ to the truncated loop soup $\mathcal{L}_{\alpha}^{\leq L}$ for $L$ large enough. The main novelty is the following lemma which controls the size of big loops and finite clusters attached to them.
\begin{lemma}\label{lem: big loop plus small finite cluster}
 For $d\geq 3$ and $\alpha>\alpha_c(d)$, for $\epsilon>0$, there exist $\widetilde{L}=\widetilde{L}(d,\alpha,\epsilon)$, $c=c(d,\alpha,\epsilon)>0$ and $C=C(d,\alpha,\epsilon)<\infty$ such that $\forall L\geq \widetilde{L}$ and $\forall n\geq 1$,
 \begin{equation}\label{eq: lblpsfc1}
  \mathbb{P}\left[\#\left\{x\in B(n) : x\overset{\mathcal{L}_{\alpha}^{\leq L}}{\centernot{\longleftrightarrow}}\infty,\exists \ell\in\mathcal{L}_{\alpha}^{>L}\text{ such that }x\overset{\mathcal{L}_{\alpha}^{\leq L}}{\longleftrightarrow}\ell\right\}>\epsilon n^d\right]<Ce^{-cn^{\frac{d}{d+1}}}.
 \end{equation}
\end{lemma}
 We will explain the way to deduce Theorem~\ref{thm: good balls} from conditions \textbf{P1}, \textbf{S1}, Theorem~\ref{thm: truncated loop percolation} and Lemma~\ref{lem: big loop plus small finite cluster} in Subsection~\ref{sect: good balls}. We postpone the verification of the condition \textbf{S1} in Subsection~\ref{sect: S1}. We prove Lemma~\ref{lem: big loop plus small finite cluster} in Subsection~\ref{sect: proof of lemma big loop}. We prove two preliminary lemmas in Subsection~\ref{sect: proof of lem proba HKaL} and \ref{sect: proof of lem isoperimetric}, which are used as intermediate steps in the proof of Theorem~\ref{thm: good balls}.

\subsection{Proof of Theorem~\ref{thm: good balls}}\label{sect: good balls}
In this subsection, we introduce necessary notation, explain the proof strategy, state two preliminary lemmas and end the subsection by proving Theorem~\ref{thm: good balls}. One preliminary lemma is probabilistic which asserts that certain event $\mathcal{H}_{K,s}^{\alpha,L}$ happens with overwhelming probabilities and the other one is deterministic which states several geometric properties implied by that event, see Lemmas~\ref{lem: proba HKaL} and \ref{lem: isoperimetric}. The proof of the preliminary lemmas are postponed in Subsections~\ref{sect: proof of lem proba HKaL} and \ref{sect: proof of lem isoperimetric}.

We closely follow the notation and strategy in \cite{Sapozhnikov2014}. First, we give a brief explanation, where the notation will be precisely defined later in details. We suppose $\alpha>\alpha_c$ and choose $L$ large enough such that there exists a unique infinite cluster $\mathcal{S}_{\infty}^{\leq L}$ of truncated loops. With a suitable multi-scale renormalization applied to the truncated loop soup $\mathcal{L}_{\alpha}^{\leq L}$ in a big box, we identify certain event $\mathcal{H}_{K,s}^{\alpha,L}$ which happens with overwhelming probabilities (see Definition~\ref{defn: H L} and Lemma~\ref{lem: proba HKaL}) and a subset of good boxes $\mathcal{Q}_{K,s,L_0}^{\leq L}$ with nice connection properties, volume lower bounds and isoperimetric inequalities (see Lemma~\ref{lem: perforated}). Each box in $\mathcal{Q}_{K,s,L_0}^{\leq L}$ of scale length $L_0$ contains a unique macroscopic percolation cluster $\mathcal{C}_x^{\leq L}$ of truncated loops. By construction, the set $\cup_{x\in\mathcal{Q}_{K,s,L_0}^{\leq L}}\mathcal{C}_x^{\leq L}$ is contained in a giant connected component $\widetilde{\mathcal{C}}_{K,s,L_0}^{\leq \infty}$ of $\mathcal{L}_{\alpha}$. Later, we will use $\widetilde{\mathcal{C}}_{K,s,L_0}^{\leq \infty}$ as $\mathcal{C}_{B_{G}(x,r)}$ in Definition~\ref{defn: very regular balls}. Then, by Definition~\ref{defn: very regular balls}, we need basically prove three things: graph distance upper bound, volume lower bound and isoperimetric inequalities on $\widetilde{\mathcal{C}}_{K,s,L_0}^{\leq \infty}$, see \eqref{eq: li2}, \eqref{eq: li4} and \eqref{eq: li5}.

The giant component $\widetilde{\mathcal{C}}_{K,s,L_0}^{\leq \infty}$ is by definition inside the box $[-2L_s,(K+2)L_s)^d$ and is part of the unique infinite cluster if $0$ belongs to the infinite cluster. Moreover, by \textbf{S1}, $\widetilde{\mathcal{C}}_{K,s,L_0}^{\leq \infty}$ will be connected locally to the frame $\cup_{x\in\mathcal{Q}_{K,s,L_0}^{\leq L}}\mathcal{C}_x^{\leq L}$ inside the infinite cluster. In the definition of $\widetilde{\mathcal{C}}_{K,s,L_0}^{\leq \infty}$ in \cite{Sapozhnikov2014}, some special care near the boundary of the box $[-2L_s,(K+2)L_s)^d$ is taken to ensure that the local connection to the frame is indeed inside $\widetilde{\mathcal{C}}_{K,s,L_0}^{\leq \infty}$, see \eqref{eq: defn tilde C L}. Then, we deduce the lower bound on graph distance and lower bound in volume of $\widetilde{\mathcal{C}}_{K,s,L_0}^{\leq \infty}$ from that of $\mathcal{Q}_{K,s,L_0}^{\leq L}$. To deduce the isoperimetric inequality for $A\subset \widetilde{\mathcal{C}}_{K,s,L_0}^{\leq \infty}$ with $C\cdot L_s^{d(d+1)}\leq\#A\leq \frac{1}{2}\cdot\widetilde{\mathcal{C}}_{K,s,L_0}^{\leq \infty}$, we use the isoperimetric inequality for the set $\{x\in\mathcal{Q}_{K,s,L_0}^{\leq L}:\mathcal{C}_x^{\leq L}\subset A\}$ inside $\mathcal{Q}_{K,s,L_0}^{\leq L}$. We will need Lemma~\ref{lem: big loop plus small finite cluster} for the upper bound of $\#\widetilde{\mathcal{C}}_{K,s,L_0}^{\leq\infty}$.

Next, we give precise definitions of necessary notation. For $L\geq 1$, we denote by $\mathcal{S}_{\infty}^{\leq L}$ the unique infinite cluster formed by the loop soup $\mathcal{L}_{\alpha}^{\leq L}$. Let $l_0,r_0$ and $L_0$ be positive integers. Let
 \[l_n=l_0\cdot 4^n,\quad r_n=r_0\cdot 2^n,\quad L_n=l_{n-1}\cdot L_{n-1}=L_0\cdot 2^{n^2}\cdot\left(\frac{l_0}{2}\right)^{n},\quad n\geq 1,\]
 where $(L_n)_n$ is a sequence of rapidly growing scales of boxes and $r_{n-1}L_{n-1}$ are the scales of bad regions in boxes of scale $L_n$ in the renormalization argument. Suppose that $l_0$ is divisible by $r_0$ and $l_0/r_0$ is large enough that
 \begin{equation}\label{eq: pf of good balls asp1}
  r_0/l_0\leq 10^{-13}\min(\theta(\alpha),2^{2-d}),
 \end{equation}
 which implies \cite[Eq. (2.7), (2.19) (3.2)]{Sapozhnikov2014}, the assumption of \cite[Lemma~3.15]{Sapozhnikov2014} and the following inequality
 \begin{equation}\label{eq: product r l lower}
  \prod\limits_{i=0}^{\infty}\left(1-\left(\frac{4r_i}{l_i}\right)^d\right)>1-10^{-12}\theta(\alpha).
 \end{equation}
 We also suppose that $L_0$ is large enough such that
\begin{equation}\label{eq: pf of good balls asp3}
 \theta(\alpha)L_0^d\geq 100.
\end{equation}
 Let $K\geq 1$ be a positive integer. For $s\geq 0$, let $\mathbb{G}_{s}=L_s\cdot\mathbb{Z}^d$ and we give $\mathbb{G}_s$ a graph structure by adding edges between $x,y$ if $||x-y||_1=L_s$. For $x\in\mathbb{Z}^d$, we define a box
 \[Q_{K,s}(x)=x+\mathbb{Z}^d\cap[0,KL_s)^d.\]
 Let $\mathcal{S}_{L_0}^{\leq L}$ be the union of clusters of diameters at least $L_0$ formed by $\mathcal{L}_{\alpha}^{\leq L}$ and let $\mathcal{S}^{\leq L}$ be the union of clusters formed by $\mathcal{L}_{\alpha}^{\leq L}$. For $x\in\mathbb{Z}^d$, let $\mathcal{C}_{K,s,L_0}^{\leq L}(x)$ be the biggest connected component in $(x+[0,KL_s)^d)\cap\mathcal{S}_{L_0}^{\leq L}$. Let $\mathcal{E}_{K,s,L_0}^{\leq L}(x)$ be the set of vertices that are connected locally to $\mathcal{C}_{K,s,L_0}^{\leq L}(x)$:
\[\mathcal{E}_{K,s,L_0}^{\leq L}(x)\overset{\text{def}}{=}\{y':\exists y\in \mathcal{C}_{K,s,L_0}^{\leq L}(x),y\overset{B(y,2L_s),\mathcal{L}_{\alpha}^{\leq L}}{\longleftrightarrow}y'\}.\]
Then, we add the set $\mathcal{E}_{K,s,L_0}^{\leq L}(x)$ to $\mathcal{C}_{K,s,L_0}^{\leq L}(x)$ and define
\begin{equation}\label{eq: defn tilde C L}
 \widetilde{\mathcal{C}}_{K,s,L_0}^{\leq L}(x)\overset{\text{def}}{=}\mathcal{C}_{K,s,L_0}^{\leq L}(x)\cup\mathcal{E}_{K,s,L_0}^{\leq L}(x).
\end{equation}
As we have mentioned before, this special care near the boundary is taken such that the open path which connects locally two vertices of $\widetilde{\mathcal{C}}_{K,s,L_0}^{\leq L}(x)$ is contained in $\widetilde{\mathcal{C}}_{K,s,L_0}^{\leq L}(x)$. Later, when we verify a ball $B_{\mathcal{S}_{\infty}}(x,r)$ is regular according to Definition~\ref{defn: very regular balls}, we will use $\widetilde{\mathcal{C}}_{K,s,L_0}^{\leq L}(x')$ as $\mathcal{C}_{B_{\mathcal{S}_{\infty}}(x,r)}$ for some suitable chosen $x',K$ and $s$. For $x=0$, we write respectively $\mathcal{C}_{K,s,L_0}^{\leq L}$, $\mathcal{E}_{K,s,L_0}^{\leq L}$ and $\widetilde{\mathcal{C}}_{K,s,L_0}^{\leq L}$ for $\mathcal{C}_{K,s,L_0}^{\leq L}(0)$, $\mathcal{E}_{K,s,L_0}^{\leq L}(0)$ and $\widetilde{\mathcal{C}}_{K,s,L_0}^{\leq L}(0)$. 

Next, we define the notion of good events, which is used to identify a frame $\cup_{\mathcal{Q}_{K,s,L_0}^{\leq L}}\mathcal{C}_{x}$. A vertex $x\in \mathbb{G}_0$ is called a \emph{$0$-good vertex} (with respect to $\mathcal{L}_{\alpha}^{\leq L}$) if
\begin{itemize}
 \item $\#(\mathcal{S}_{L_0}^{\leq L}\cap(x+[0,L_0)^d))<\frac{11}{10}\theta^{(L)}(\alpha)L_0^d$,
 \item for each $y\in\mathbb{G}_0$ with $||y-x||_1\leq L_0$, the set $\mathcal{S}_{L_0}^{\leq L}\cap(y+[0,L_0)^d)$ contains a connected component $\mathcal{C}_y^{\leq L}$ with at least $\frac{9}{10}\theta^{(L)}(\alpha)L_0^d$ vertices such that for all $y\in\mathbb{G}_0$ with $||y-x||_1\leq L_0$, $\mathcal{C}_y^{\leq L}$ and $\mathcal{C}_x^{\leq L}$ are connected in $\mathcal{S}^{\leq L}\cap((x+[0,L_0)^d))\cup((y+[0,L_0)^d))$, where $\mathcal{S}^{\leq L}$ are the clusters formed by $\mathcal{L}_{\alpha}^{\leq L}$. 
\end{itemize}
Recursively, we define a vertex $x\in\mathbb{G}_{n+1}$ to be \emph{$(n+1)$-good} if there does not exist any pair of vertices $(x_1,x_2)$ such that $||x_1-x_2||_{\infty}\geq r_{n}L_n$ and $x_1,x_2$ are \emph{$n$-bad} (i.e. they are not $n$-good). Note that this is slightly different from the definition used in \cite{Sapozhnikov2014}, where there may have two $n$-bad vertices in the definition of $(n+1)$-good vertices, see \cite[Section~2.1,3.1]{Sapozhnikov2014}. Their only tool to control the correlations between boxes is the decoupling inequalities for monotone events. But we have independence for boxes with distance bigger than $L$ for truncated loop models $\mathcal{L}_{\alpha}^{\leq L}$. For that reason, we use a simplified definition. Moreover, $(r_n)_n$ need not be growing and \textbf{S2} is not needed. Also, note that one could use the notion of good boxed in the sense of \cite{DeuschelPisztoraMR1384041} for the truncated model, see \cite[Section~7.4]{GrimmettMR1707339}. However, we decide to follow the notation of \cite{Sapozhnikov2014}.

Next, we identify the frame by removing the bad regions:
\begin{multline}\label{eq: defn frame}
 \mathcal{Q}_{K,s,0}^{\leq L}(x_s)=\mathbb{G}_0\cap (x_s+[-2L_s,(K+2)L_s)^d)\\
 -\cup_{n<s}\cup_{z_n\text{ is }n\text{-bad}}(z_n+[0,2r_{n}L_{n})^d)\cap(\lfloor z_n/L_{n+1}\rfloor+[0,L_{n+1})^d)),
\end{multline}
where $\lfloor z\rfloor\overset{\text{def}}{=}(\lfloor z^1\rfloor,\ldots,\lfloor z^d\rfloor)$ for $z=(z^1,\ldots,z^d)\in\mathbb{R}^d$. We didn't remove bad regions related to $s$-bad vertices since we will assume later that all vertices in $\mathbb{G}_s\cap (x_s+[-2L_s,(K+2)L_s)^d)$ are $s$-good.

Next, we define the event $\mathcal{H}_{K,s}^{\alpha,L}$ and state two preliminary lemmas.
\begin{defn}\label{defn: H L}
 For $x_s\in\mathbb{G}_s$, the event $\mathcal{H}^{\alpha,L}_{K,s}(x_s)$ occurs if
 \begin{itemize}
  \item [a)] all the vertices in $\mathbb{G}_{s}\cap (x_s+[-2L_s,(K+2)L_s)^d)$ are $s$-good with respect to $\mathcal{L}_{\alpha}^{\leq L}$,
  \item [b)] Any $x,y\in\mathcal{S}_{L_s}\cap (x_s+[0,KL_s)^d)$ with $||x-y||_{\infty}\leq L_s$, we have that $x\overset{B(x,2L_s),\mathcal{L}_{\alpha}}{\longleftrightarrow}y$.
  \item [c)] For all $\vec{k}\in\{-2,\ldots,K+1\}^d$, we have that
 \[\#\left\{x\in B^{(L_s)}(\vec{k}) : x\overset{\mathcal{L}_{\alpha}^{\leq L}}{\centernot{\longleftrightarrow}}\infty,\exists \ell\in\mathcal{L}_{\alpha}^{>L}\text{ such that }x\overset{\mathcal{L}_{\alpha}^{\leq L}}{\longleftrightarrow}\ell\right\}\leq \frac{\theta(\alpha)}{100} L_s^d.\]
 \end{itemize}
 We write $\mathcal{H}_{K,s}^{\alpha,L}$ for $\mathcal{H}_{K,s}^{\alpha,L}(0)$.
\end{defn}

For the probability of the event $\mathcal{H}^{\alpha,L}_{K,s}$, we have
\begin{lemma}\label{lem: proba HKaL}
 There exists $\tilde{L}=\tilde{L}(d,\alpha)<\infty$ such that for $L>\tilde{L}$, there exist $C(d,\alpha,L)<\infty$, $c(d,\alpha,L)>0$ and $C'=C'(d,\alpha,l_0,L)$ such that for $r_0\geq C$, $l_0\geq Cr_0$ and $L_0>C'$, we have that
\begin{equation}\label{eq: proba tilde H}
 \mathbb{P}[\mathcal{H}_{K,s}^{\alpha,L}]\geq 1-CK^d(2^{-2^s}+L_s^{2d}e^{-cL_s}+e^{-cL_s^{\frac{d}{d+1}}}).
\end{equation}
\end{lemma}
The event $\mathcal{H}^{\alpha,L}_{K,s}(x_s)$ implies nice geometric properties of $\widetilde{\mathcal{C}}_{K,s,L_0}^{\leq \infty}(x_s)\cap B^{(L_s)}(\vec{k})$, see
\begin{lemma}\label{lem: isoperimetric}
 We suppose that $\mathcal{H}_{K,s}^{\alpha,L}(x_s)$ occurs for $x_s\in\mathbb{G}_s$ and that $\theta^{(L)}(\alpha)>\frac{99}{100}\theta(\alpha)$. (By Theorem~\ref{thm: truncated loop percolation} and Remark~\ref{rem: theta L to theta}, we have $\theta^{(L)}(\alpha)>\frac{99}{100}\theta(\alpha)$ for $L$ large enough.) Then, we have the following properties:
 \begin{itemize}
  \item[a)] Volume control: for $\vec{k}\in\{-2,\ldots,K+1\}^d$, we have that
  \begin{equation}\label{eq: li2}
   \#\left(\widetilde{\mathcal{C}}_{K,s,L_0}^{\leq \infty}(x_s)\cap B^{(L_s)}(\vec{k})\right)\geq \frac{8}{11}\theta(\alpha)L_s^d.
  \end{equation}
  \item[b)] Connectivity property: for $m\geq 1$ and $x,y\in\widetilde{\mathcal{C}}_{K,s,L_0}^{\leq\infty}(x_s)$ with $||x-y||_{\infty}\leq mL_s$,
  \begin{equation}\label{eq: li3.5}
   x\text{ and }y\text{ are connected by open edges inside }\widetilde{\mathcal{C}}_{L,s,L_0}^{\leq\infty}(x_s)\cap B(x,(16+m)L_s).
  \end{equation} There exists $C=C(L_0)$ such that for all $y,y'\in \widetilde{\mathcal{C}}_{K,s,L_0}^{\leq\infty}(x_s)$,
  \begin{equation}\label{eq: li4}
   d_{\mathcal{S}}(y,y')\leq C\max(||y-y'||_{\infty},L_s^d).
  \end{equation}
  \item[c)] Isoperimetric inequality: $\forall \delta>0$, there exist $C=C(d,\alpha)$ and  $\gamma=\gamma(d,L_0,\delta)$ such that for $A\subset \widetilde{\mathcal{C}}_{K,s,L_0}^{\leq\infty}(x_s)$ with $CL_s^{d(d+1)}\leq \#A\leq (1-\delta)\#\widetilde{\mathcal{C}}_{K,s,L_0}^{\leq\infty}(x_s)$,
  \begin{equation}\label{eq: li5}
   \#\partial_{\widetilde{\mathcal{C}}_{K,s,L_0}^{\leq\infty}(x_s)}A\geq \gamma\cdot(\#A)^{\frac{d-1}{d}}.
  \end{equation}
 \end{itemize}
\end{lemma}
We postpone the proof of Lemma~\ref{lem: proba HKaL} in Subsection~\ref{sect: proof of lem proba HKaL} and the proof of Lemma~\ref{lem: isoperimetric} in Subsection~\ref{sect: proof of lem isoperimetric}. Finally, we end this subsection by deducing Theorem~\ref{thm: good balls} from the preliminary lemmas, namely, Lemmas \ref{lem: proba HKaL} and \ref{lem: isoperimetric}.

\begin{proof}[Proof of Theorem~\ref{thm: good balls}]
 We will prove the theorem for $\epsilon=\frac{1}{2(d+2)}$. We assume in the following context that $\mathcal{H}_{2K,s}^{\alpha,L}(x_s)$ occurs, where $x_s=(-KL_s,\ldots,-KL_s)$ and $K=\lceil R/L_s\rceil+1$. By the property $b)$ in the definition of $\mathcal{H}_{2K,s}^{\alpha,L}$, when $0\in\mathcal{S}_{\infty}$, we have that $0\in\widetilde{\mathcal{C}}_{K,s,L_0}^{\leq\infty}\subset \mathcal{S}_{\infty}$. Also, note that the ball $B_{\mathcal{S}_{\infty}}(0,R)$ is contained in the box $B(0,R)=[-R,R]^d\subset\mathbb{Z}^d$. We take $s=\left\lceil(\log R)^{1/2}(2d(d+1)(d+2))^{-1/2}\right\rceil$, then there exists $R_0(d,l_0,L_0)<\infty$ such that for $R\geq R_0$, 
  \[L_s=L_0\cdot 2^{s^2}\left(\frac{l_0}{2}\right)^{s}=L_0\cdot R^{\frac{\log 2}{2d(d+1)(d+2)}+o_{R}(1)}\left(\frac{l_0}{2}\right)^{\sqrt{\frac{\log R}{2d(d+1)(d+2)}}+o_R(1)}\leq R^{\frac{1}{2d(d+1)(d+2)}}.\]
  Next, we take $N_{\mathcal{S}_{\infty}(0,R)}=L_s^{d(d+1)}$. Then, for $R\geq R_0$, $N_{\mathcal{S}_{\infty}}(0,R)\leq R^{\epsilon}$. Let $r\in[N_{\mathcal{S}_{\infty}(0,R)},R]$ and consider $B_{\mathcal{S}_{\infty}}(y,r)\subset B_{\mathcal{S}_{\infty}}(0,R)$ for $y\in\mathbb{Z}^d$. Let $K'=\lceil\frac{r}{L_s}\rceil+1$. Then, we have that $K'\geq L_s^{d(d+1)-1}$. We take $y_s\in\mathbb{G}_s$ such that $y\in y_s+[0,L_s)^d$. Then, we have that
  \[B_{\mathcal{S}_{\infty}}(y,r)\subset B(y_s,K'L_s))\subset B(0,KL_s).\]
  We take $\mathcal{C}_{B_{\mathcal{S}_{\infty}}(y,r)}$ to be $\widetilde{\mathcal{C}}_{2K',s,L_0}^{\leq\infty}(z_s)$ where $z_s=y_s-(K'L_s,\ldots,K'L_s)$. By \eqref{eq: li2}, there exists $y'\in (y_s+[0,L_s)^d)\cap \widetilde{\mathcal{C}}_{2K',s,L_0}^{\leq\infty}(z_s)$. Hence, by the property $b)$ of the event $\mathcal{H}_{2K,s}^{\alpha,L}(x_s)$, \[B_{\mathcal{S}_{\infty}}(y,r)\subset \mathcal{C}_{B_{\mathcal{S}_{\infty}}(y,r)}.\]
  By graph distance bounds \eqref{eq: li4} in $\mathcal{C}_{B_{\mathcal{S}_{\infty}}(y,r)}$, there exists $C_W=C_W(L_0)<\infty$ such that
  \[\mathcal{C}_{B_{\mathcal{S}_{\infty}}(y,r)}\subset B_{\mathcal{S}_{\infty}}(y,C_Wr).\]
  Hence, by \eqref{eq: li2}, there exists $c=c(d,\alpha,L_0)>0$ such that
  \[\#B_{\mathcal{S}_{\infty}}(y,r)\geq cr^d\]
  By \eqref{eq: li5}, for $A\subset \mathcal{C}_{B_{\mathcal{S}_{\infty}}(y,r)}$ with $\#A\leq\frac{1}{2}\cdot\#\mathcal{C}_{B_{\mathcal{S}_{\infty}(y,r)}}$, there exists $C_P(d,\alpha,L_0)<\infty$ such that
  \[\#\partial_{\mathcal{C}_{B_{\mathcal{S}_{\infty}}(y,r)}}A\geq \frac{\#A}{r\sqrt{C_P}}.\]
  Finally, the proof is complete by \eqref{eq: proba tilde H}.
\end{proof}

\subsection{Local connection property \textbf{S1}}\label{sect: S1}
In this section, we verify the condition \textbf{S1} by using Theorem~\ref{thm: truncated loop percolation}, Corollary~\ref{cor: exponential tail} and following the argument in \cite[Lemma~7.89]{GrimmettMR1707339}. We split \textbf{S1} into the following two lemmas and prove them separately.
\begin{lemma}\label{lem: connection from a big box}
 For $d\geq 3$ and $\alpha>\alpha_c(d)$, there exist $L_0=L_0(d,\alpha)<\infty$ and $c=c(d,\alpha)>0$ such that $\forall n\geq L_0(d,\alpha)$,
 \begin{equation}
  \mathbb{P}\left[B(n)\overset{\mathcal{L}_{\alpha}^{\leq L_0}}{\longleftrightarrow}\infty\right]\geq 1-e^{-cn^{d-2}}.
 \end{equation}
\end{lemma}
\begin{lemma}\label{lem: local connectness}
 For $d\geq 3$, $\alpha>\alpha_c(d)$ and $b>1$, there exist $L_0=L_0(d,\alpha)$ and $c_0=c_0(d,\alpha,b)$ such that $\forall L\geq L_0$, $n\geq 1$ and $x,y\in B(n)$,
 \begin{equation}\label{eq: llc1}
  \mathbb{P}\left[x\overset{\mathcal{L}_{\alpha}^{\leq L}}{\longleftrightarrow}\partial B(bn),y\overset{\mathcal{L}_{\alpha}^{\leq L}}{\longleftrightarrow}\partial B(bn),x\overset{B(bn),\mathcal{L}_{\alpha}^{\leq L}}{\centernot{\longleftrightarrow}}y\right]\leq e^{-nc_0}.
 \end{equation}
\end{lemma}
We first verify Lemma~\ref{lem: connection from a big box} which asserts that a big box is necessarily connected to infinity in the supercritical regime.
\begin{proof}[Proof of Lemma~\ref{lem: connection from a big box}]
 By taking $L_0=L_0(d,\alpha)$ large enough, for all $\vec{K}=(K_3,\ldots,K_d)\in\mathbb{Z}^{d-2}$, there is a percolation inside a two dimensional slab
 \[Slab(\vec{K})=\mathbb{Z}^2\times\{K_3L_0,\ldots,(K_3+1)L_0-1\}\times\cdots\times \{K_dL_0,\ldots,(K_d+1)L_0-1\}\]
 of width $L_0$ by loops $(\mathcal{L}_{\alpha}^{\leq L_0})^{Slab(\vec{K})}$ inside the slab of length no more than $L_0$. Since
 \[\mathbb{P}\left[B(n)\overset{\mathcal{L}_{\alpha}^{\leq L_0}}{\longleftrightarrow}\infty\right]\geq \mathbb{P}\left[\exists \vec{K}\in\mathbb{Z}^{d-2}\text{ such that }B(n)\cap Slab(\vec{K})\overset{(\mathcal{L}_{\alpha}^{\leq L_0})^{Slab(\vec{K})}}{\longleftrightarrow}\infty\right],\]
 this lemma follows from the translation invariance and the independence between disjoint sets of loops inside disjoint slabs.
\end{proof}

Next, we prove Lemma~\ref{lem: local connectness} by following the same strategy in the proof \cite[Lemma 7.89]{GrimmettMR1707339} for Bernoulli percolations.

\begin{proof}[Proof of Lemma~\ref{lem: local connectness}]
 We follow the argument in \cite[Lemma~7.89]{GrimmettMR1707339}. Without loss of generality, we assume that $x,y\in\partial B(n)$. Let $L_0=L_0(d,\alpha)$ be a large enough number, which will be specified later in the proof. We take a sequence of numbers $N_j=n+jL_0$ for $j=0,\ldots,J_0$ where $J_0=\lfloor(b-1)n/L_0\rfloor$. Consider events $A_0\supset A_1\supset\cdots\supset A_{J_0}$, where
 \[A_j\overset{\text{def}}{=}\{x\overset{\mathcal{L}_{\alpha}^{\leq L}}{\longleftrightarrow}\partial B(N_j),y\overset{\mathcal{L}_{\alpha}^{\leq L}}{\longleftrightarrow}\partial B(N_j),x\overset{B(N_j-1),\mathcal{L}_{\alpha}^{\leq L}}{\centernot\longleftrightarrow}y\}.\]
 Then,
 \[\eqref{eq: llc1}\leq \mathbb{P}[A_j]=\mathbb{P}[A_0]\prod\limits_{j=0}^{J_0-1}\mathbb{P}[A_{j+1}|A_{j}]\leq \prod\limits_{j=0}^{J_0-1}\mathbb{P}[A_{j+1}|A_{j}].\]
 Note that $A_j$ is a measurable function of $(\mathcal{L}_{\alpha})_{B(N_j-1)}$ for each $j$. Also, on $A_j$, we can choose two vertices $x_j$ and $y_j$ according to lexicographic order such that $x_j,y_j\in\partial B(N_j)$, $x\overset{B(N_j),(\mathcal{L}_{\alpha}^{\leq L})_{B(N_j-1)}}{\longleftrightarrow}x_j$ and $y\overset{B(N_j),(\mathcal{L}_{\alpha}^{\leq L})_{B(N_j-1)}}{\longleftrightarrow}y_j$. Then, we bound the probability $\mathbb{P}[A_{j+1}|A_j]$ from above:
 \begin{align*}\mathbb{P}[A_{j+1}|A_j]\leq & \mathbb{P}\Big[x\overset{B(N_{j+1}-1),\mathcal{L}_{\alpha}^{\leq L}}{\centernot{\longleftrightarrow}}y\Big|A_j\Big]\\
 \leq &\frac{1}{\mathbb{P}[A_j]}\mathbb{P}\Big[x_j\overset{(\mathcal{L}_{\alpha}^{\leq L})^{B(N_j-1)^c}_{B(N_{j+1}-1)}}{\centernot{\longleftrightarrow}}y_j,A_j\Big]\\
 \leq &\sup\limits_{u,v\in\partial B(N_j)}\mathbb{P}\Big[u\overset{(\mathcal{L}_{\alpha}^{\leq L})^{B(N_j-1)^c}_{B(N_{j+1}-1)}}{\centernot{\longleftrightarrow}}v\Big],
 \end{align*}
 where we use the independence between $(\mathcal{L}_{\alpha}^{\leq L})_{B(N_j-1)}$ and $(\mathcal{L}_{\alpha}^{\leq L})^{B(N_j-1)^c}_{B(N_{j+1}-1)}$ in the last step. For \eqref{eq: llc1}, it suffices to have a uniform lower bound of $\mathbb{P}\Big[u\overset{(\mathcal{L}_{\alpha}^{\leq L})^{B(N_j-1)^c}_{B(N_{j+1})}}{\longleftrightarrow}v\Big]$ for $L_0$ large enough, $L\geq L_0$, $j=0,1,\ldots,J_0$ and $u,v\in\partial B(N_j)$. For this part, we basically follow \cite[Lemma~7.78]{GrimmettMR1707339}. By FKG inequality, it suffices to prove that $\exists L_0<\infty$ and $c=c(d,\alpha)$ such that for all $n\geq 1$,
 \begin{equation}\label{eq: lcc2}
  \mathbb{P}\left[(0,0,\ldots,0)\overset{(\mathcal{L}_{\alpha}^{\leq L_0})^{Q_n}}{\longleftrightarrow}(n,0,\ldots,0)\right]>c,
 \end{equation}
 where $Q_n$ denote the cuboid $\{0,\ldots,n\}^2\times\{0,\ldots,L_0-1\}^{d-2}$. By Theorem~\ref{thm: truncated loop percolation} and FKG inequality, we may take $L_0$ large enough such that for $\alpha>\alpha_c$,
 \begin{equation*}
  \mathbb{P}\Big[(0,\ldots,0)\overset{(\mathcal{L}_{\alpha'}^{\leq L_0})^{\mathbb{Z}_{+}^{2}\times\{0,\ldots,L_0-1\}^{d-2}}}{\longleftrightarrow}\infty\Big]>0,
 \end{equation*}
 where $\alpha'=\frac{\alpha+\alpha_c}{2}$. Then, by FKG inequality and definition of the loop soup, we have that
 \[\inf\limits_{n}\mathbb{P}\Big[(0,\ldots,0)\overset{(\mathcal{L}_{\alpha'}^{\leq L_0})^{Q_n}}{\longleftrightarrow}P(n)\Big]>0,\]
 where $P(m)\overset{\text{def}}{=}\{m\}\times\{0,\ldots,m\}\times\{0,\ldots,L_0-1\}^{d-2}$ for $m=0,\ldots,n$. By symmetry and FKG inequality, we also have that
 \[\inf\limits_{n}\mathbb{P}\Big[(0,0,\ldots,0)\overset{(\mathcal{L}_{\alpha'}^{\leq L_0})^{Q_n}}{\longleftrightarrow}P(n),(n,0,\ldots,0)\overset{(\mathcal{L}_{\alpha'}^{\leq L_0})^{Q_n}}{\longleftrightarrow}P(0)\Big]>0.\]
 When the event $(0,0,\ldots,0)\overset{(\mathcal{L}_{\alpha'}^{\leq L_0})^{Q_n}}{\longleftrightarrow}P(n),(n,0,\ldots,0)\overset{(\mathcal{L}_{\alpha'}^{\leq L_0})^{Q_n}}{\longleftrightarrow}P(0)$ occurs, there are two paths connecting $(0,\ldots,0)$ to $P(n)$ and $(n,0,\ldots,0)$ to $P(0)$. Necessarily, their projections on the first two coordinates intersect at some vertex. We choose the smallest such vertex $(w_1,w_2)$ according to the lexicographic order. Note that there exists $c=c(d,\alpha)>0$ such that
 \[\mathbb{P}[\text{All edges in }\{w_1,w_2\}\times\{0,\ldots,L_0-1\}^{d-2}\text{ are covered by some loop in }(\mathcal{L}_{\alpha-\alpha'}^{\leq L_0})^{Q_n}]>c.\]
 Finally, \eqref{eq: lcc2} follows by the independence and stationarity of $\alpha\rightarrow(\mathcal{L}_{\alpha}^{\leq L_0})^{Q_n}$.
\end{proof}

\subsection{Proof of Lemma~\ref{lem: big loop plus small finite cluster}}\label{sect: proof of lemma big loop}
We give the proof of Lemma~\ref{lem: big loop plus small finite cluster} in the present section. It is an estimate of the finite clusters attached to some big loops within a large box. We first give an estimate of the finite clusters in Lemma~\ref{lem: mean cluster size bound} as a corollary of Corollary~\ref{cor: exponential tail}. Next, we give an upper bound of the vertices covered by the big loops in Lemma~\ref{lem: vertices visited by large loops}. Finally, we prove Lemma~\ref{lem: big loop plus small finite cluster} by combining them together.

As an immediate consequence of Corollary~\ref{cor: exponential tail}, we have 
\begin{lemma}\label{lem: mean cluster size bound}
 For $d\geq 3$ and $\alpha>\alpha_c(d)$, there exist $L_0=L_0(d,\alpha)$ and $C=C(d,\alpha)<\infty$ such that $\forall L\geq L_0$,
 \begin{equation}
  \mathbb{E}\left[\#\mathcal{C}_{\alpha}(0),\#\mathcal{C}_{\alpha}(0)<\infty\right]<C(d,\alpha).
 \end{equation}
\end{lemma}

\begin{lemma}\label{lem: vertices visited by large loops}
 For $d\geq 3$ and $\alpha>0$, for all $\epsilon>0$, there exist $L_0=L_0(d,\alpha,\epsilon)<\infty$, $C=C(d,\alpha,\epsilon)<\infty$ and $c=c(d,\alpha,\epsilon)>0$ such that for $L\geq L_0$ and for all $n\geq 1$, 
 \begin{equation}\label{eq: lvvbll1}
  \mathbb{P}\left[\sum\limits_{x\in B(n)}1_{\{\exists\ell\in\mathcal{L}_{\alpha}:x\in\ell,|\ell|\geq L\}}\geq \epsilon n^d\right]\leq Ce^{-cn^{d-2}}.
 \end{equation}
\end{lemma}
\begin{remark}
 We see that the exponent $n^{d-2}$ is optimal for small enough $\epsilon$ by considering the probability that there exist $O(n^{d-2})$ many loops of diameter $n$ intersecting $B(n)$. Each such loop appears independently with a positive probability bounded from below. By Paley-Zygmund inequality, those loops occupy $O(n^d)$ of vertices inside $B(n)$.
 
 Besides, by using Hoeffding's inequality for martingales (see \cite{FreedmanMR0380971} or \cite[Equations (10) and (15)]{FanGramaLiuMR2956116}), one can show that for $d\geq 3$ and $\alpha>0$, for all $t>0$, there exist $C=C(d,\alpha)<\infty$ and $c=c(d,\alpha)>0$ such that for all intensity function $\beta$ bounded by $\alpha$ (i.e. $\sup\limits_{v}|\beta(v)|\leq\alpha$) and for all $n\geq 1$, 
 \begin{equation}\label{eq: lvvbll1a}
  \mathbb{P}\left[\left|N(\mathcal{L}_{\beta},B(n))-\mathbb{E}[N(\mathcal{L}_{\beta},B(n))]\right|\geq t\right]\leq Ce^{-cf(d,n,t)},
 \end{equation}
 where $N(\mathcal{L}_{\beta},B(n))\overset{\text{def}}{=}\sum\limits_{x\in B(n)}1_{\{\exists\ell\in\mathcal{L}_{\beta}:x\in\ell\}}$ and $f(d,n,t)=\frac{t}{n^2}1_{d=3}+\frac{t^2}{n^4\log n}1_{d=4}+\min(\frac{t^2}{n^d},\frac{t}{n^2})1_{d\geq 5}$. To get \eqref{eq: lvvbll1} from \eqref{eq: lvvbll1a}, we take $\mathcal{L}_{\beta}=\{\ell\in\mathcal{L}_{\alpha}:|\ell|\geq L\}$. As we do not need \eqref{eq: lvvbll1a}, we omit the proof.
\end{remark}

\begin{proof}
 Consider another based loop measure $\widetilde{\dot{\mu}}=\widetilde{\dot{\mu}}_n$ on the based loops visiting $B(n)$ as follows: for $\dot{\ell}=(x_1,\ldots,x_k)$ ($k\geq 2$),
 \begin{equation}\label{eq: lvvbll2}
  \widetilde{\dot{\mu}}(\dot{\ell})=\frac{1}{N(\dot{\ell},B(n))}1_{\{x_1\in B(n)\}}Q^{x_1}_{x_2}\ldots Q^{x_{k-1}}_{x_k}Q^{x_k}_{x_1},
 \end{equation}
 where $N(\dot{\ell},B(n))\overset{\text{def}}{=}\sum\limits_{i=1}^{k}1_{\{x_i\in B(n)\}}$. Then, $\widetilde{\dot{\mu}}$ and $\dot{\mu}$ induces the same loop measure on the loops visiting $B(n)$. Let $\dot{\mathcal{B}}_{\alpha}=\{(t,\Base(\dot{\ell})):(t,\dot{\ell})\in\widetilde{\dot{\mathcal{LP}}},t\leq\alpha\}$. Then, $\dot{\mathcal{B}}_{\alpha}$ is a Poisson point process on $[0,\alpha]\times B(n)$ of the intensity measure $\Leb([0,\alpha])\otimes\left(\sum\limits_{y\in B(n)}\sum\limits_{\Base(\dot{\ell})=y}\widetilde{\dot{\mu}}(\dot{\ell})\delta_y\right)$. To recover $\widetilde{\dot{\mathcal{LP}}}$, conditionally on $\dot{\mathcal{B}}_{\alpha}$, we sample independently at each couple $(t,y)\in\dot{\mathcal{B}}_{\alpha}$ a loop based at $y$ according to the following family of probability measures $\left(\widetilde{\dot{P}}_y\right)_{y\in B(n)}$:
 \begin{equation}\label{eq: lvvbll3}
  \widetilde{\dot{P}}_y(\dot{\ell})=\frac{1_{\{\Base(\dot{\ell})=y\}}\widetilde{\dot{\mu}}(\dot{\ell})}{\sum\limits_{\Base(\dot{\ell})=y}\widetilde{\dot{\mu}}(\dot{\ell})}\text{ for }y\in B(n).
 \end{equation}
 We list $\dot{\mathcal{B}}_{\alpha}$ in the increasing order of $t$: $\dot{\mathcal{B}}_{\alpha}=\{(t_1,\Base(\dot{\ell}_1)),\ldots,(t_{M_n},\Base(\dot{\ell}_{M_n}))\}$. (The way of ordering is of little importance.) For $i=1,\ldots,M_n$, we define
 \[Y_{i}=1_{\{|\dot{\ell}_i|>L\}}\cdot\#\{x\in B(n):x\in \dot{\ell}_i\}.\]
 To prove \eqref{eq: lvvbll1}, it suffices to prove that $\forall\epsilon>0$, there exist $L_0=L_0(d,\alpha,\epsilon)<\infty$, $C=C(d,\alpha,\epsilon)<\infty$ and $c=c(d,\alpha,\epsilon)>0$ such that for $L\geq L_0$ and for all $n\geq 1$, 
 \begin{equation}\label{eq: lvvbll4}
  \mathbb{P}\left[\sum\limits_{i=1}^{M_n}Y_i\geq \epsilon n^d\right]\leq Ce^{-cn^{d-2}}.
 \end{equation}
 Note that $M_n$ is a Poisson variable of expectation $\alpha\sum\limits_{y\in B(n)}\sum\limits_{\Base(\dot{\ell})=y}\widetilde{\dot{\mu}}(\dot{\ell})$, which is bounded by $\alpha C(d)n^d$ when $d\geq 3$. Hence, there exist $C=C(d,\alpha)<\infty$ and $c=c(d,\alpha)>0$ such that
 \begin{equation}\label{eq: lvvbll5}
  \mathbb{P}[M_n>Cn^d]\leq e^{-cn^d} \quad,\forall n\geq 1.
 \end{equation}
 Next, by \eqref{eq: lvvbll2} and \eqref{eq: lvvbll3}, noting that $\exists C=C(d)<\infty$ such that $\sum\limits_{\Base(\dot{\ell})=y}\widetilde{\dot{\mu}}(\dot{\ell})<\infty$, we get that
 \[\sup\limits_{i}\mathbb{E}[e^{\lambda Y_i/n^2}-1|\dot{\mathcal{B}}_{\alpha}]\leq C\max\limits_{x\in B(n)}\sum\limits_{k>L}\sum\limits_{x_2,\ldots,x_k}Q^{x}_{x_2}\cdots Q^{x_{k-1}}_{x_k}Q^{x_k}_{x}\frac{\exp\{\lambda N(\dot{\ell},B(n))/n^2\}-1}{N(\dot{\ell},B(n))}.\]
 Note that
 \begin{align*}
  \frac{\exp\{\lambda N(\dot{\ell},B(n))/n^2\}-1}{N(\dot{\ell},B(n))}\leq & 1_{\{N(\dot{\ell},B(n))\leq \frac{\log L}{3\lambda}n^2\}}\frac{\lambda}{n^2}L^{1/3}\\
  &+1_{\{N(\dot{\ell},B(n))>\frac{\log L}{3\lambda}n^2\}}\frac{3\lambda}{n^2\log L}\exp\{\lambda N(\dot{\ell},B(n))/n^2\}.
 \end{align*}
 Thus, for $d\geq 3$, there exists $C=C(d)<\infty$ such that
 \begin{align*}
  \sup\limits_{i}\mathbb{E}[e^{\lambda Y_i/n^2}-1|\dot{\mathcal{B}}_{\alpha}]
  \leq &C\frac{\lambda}{n^2}L^{1/3}\max\limits_{x\in B(n)}\sum\limits_{k>L}\sum\limits_{x_2,\ldots,x_k}Q^{x}_{x_2}\cdots Q^{x_{k-1}}_{x_k}Q^{x_k}_{x}\\
  &+C\frac{\lambda}{n^2\log L}\max\limits_{x\in B(n)}\sum\limits_{k}\sum\limits_{x_2,\ldots,x_k}Q^{x}_{x_2}\cdots Q^{x_{k-1}}_{x_k}Q^{x_k}_{x}e^{\lambda N(\dot{\ell},B(n))/n^2}\\
  \leq & C\frac{\lambda}{n^2\log L}\left(1+\max\limits_{x\in B(n)}\sum\limits_{k}\sum\limits_{x_2,\ldots,x_k}Q^{x}_{x_2}\cdots Q^{x_{k-1}}_{x_k}Q^{x_k}_{x}e^{\lambda N(\dot{\ell},B(n))/n^2}\right).
 \end{align*}
 If $(X_k)_k$ is a simple random walk on $\mathbb{Z}^d$ starting from $x$, then $\exists C=C(d)<\infty$ such that
 \[\sum\limits_{k}\sum\limits_{x_2,\ldots,x_k}Q^{x}_{x_2}\cdots Q^{x_{k-1}}_{x_k}Q^{x_k}_{x}e^{\lambda N(\dot{\ell},B(n))/n^2}\leq C\mathbb{E}^{x}\left[\exp\left\{\frac{\lambda}{n^2}\sum\limits_{k=0}^{\infty}1_{\{X_k\in B(n)\}}\right\}\right].\]
 From an estimation of Green function for $d\geq 3$, we deduce that $\exists C=C(d)<\infty$ such that for $\lambda$ small enough,
 \[\mathbb{E}^{x}\left[\frac{\lambda}{n^2}\sum\limits_{k=0}^{\infty}1_{\{X_k\in B(n)\}}\right]\leq C\lambda<1/2,\]
 when $\lambda$ is small enough. By Khas'minskii's lemma (see e.g. \cite[Lemma~3.7]{ChungZhaoMR1329992}), there exists $C=C(d)<\infty$ such that
 \[\mathbb{E}\left[\exp\left\{\frac{\lambda}{n^2}\sum\limits_{k=0}^{\infty}1_{\{X_k\in B(n)\}}\right\}\right]\leq C\lambda.\]
 Hence, there exists $C=C(d)<\infty$ such that
 \begin{equation}\label{eq: lvvbll6}
  \sup\limits_{i}\mathbb{E}[e^{\lambda Y_i/n^2}-1|\dot{\mathcal{B}}_{\alpha}]
  \leq C\frac{\lambda}{n^2\log L}
 \end{equation}
 and \eqref{eq: lvvbll4} follows from \eqref{eq: lvvbll5} and \eqref{eq: lvvbll6}.
\end{proof}

We are ready for the proof of Lemma~\ref{lem: big loop plus small finite cluster}. Note that
\[\left\{x\in B(n) : x\overset{\mathcal{L}_{\alpha}^{\leq L}}{\centernot{\longleftrightarrow}}\infty,\exists \ell\in\mathcal{L}_{\alpha}^{>L}\text{ such that }x\overset{\mathcal{L}_{\alpha}^{\leq L}}{\longleftrightarrow}\ell\right\}\]
 is non-increasing in $L$. Hence, the constants $c,C$ are independent of $L$ and it suffices to prove \eqref{eq: lblpsfc1} for some $L$ large enough. We denote by $S$ the variable in \eqref{eq: lblpsfc1}
\[\#\left\{x\in B(n) : x\overset{\mathcal{L}_{\alpha}^{\leq L}}{\centernot{\longleftrightarrow}}\infty,\exists \ell\in\mathcal{L}_{\alpha}^{>L}\text{ such that }x\overset{\mathcal{L}_{\alpha}^{\leq L}}{\longleftrightarrow}\ell\right\}.\]
 Set $m=\lceil n^{\frac{d}{d+1}}\rceil$. Let $J=\{\vec{j}\in\mathbb{Z}^d: B^{(m)}(\vec{j})\cap B(n)\neq\emptyset\}$. For $r\geq 0$, we denote by $U_{m}^{r}(\vec{j})$ the $r$-neighborhood of $B^{(m)}(\vec{j})$ with respect to $||\cdot||_{\infty}$ distance:
 \[U_{m}^{r}(\vec{j})=\{j_1m-r,\ldots,(j_1+1)m_n+r-1\}\times\cdots\times\{j_dm-r,\ldots,(j_d+1)m+r-1\}.\]
 We define
 \[S(\vec{j})=\#\{x\in B^{(m)}(\vec{j}):x\overset{(\mathcal{L}_{\alpha}^{\leq L})^{U_{m}^{2m}(\vec{j})}}{\centernot{\longleftrightarrow}}\partial U_{m}^{m}(\vec{j}),\exists \ell\in\mathcal{L}_{\alpha}^{>L}\text{ such that }x\overset{(\mathcal{L}_{\alpha}^{\leq L})^{U_{m}^{2m}(\vec{j})}}{\longleftrightarrow}\ell \}.\]
 By taking $n$ large enough, we assume that $m\geq L$.
 When each finite cluster formed by $\mathcal{L}_{\alpha}^{\leq L}$ has diameter at most $n^{\frac{d}{d+1}}$, we have that $S\leq\sum\limits_{\vec{j}\in J}S(\vec{j})$. By Corollary~\ref{cor: exponential tail}, such event occurs with high probability: there exist $L_0=L_0(d,\alpha)<\infty$, $C=C(d,\alpha)<\infty$ and $c=c(d,\alpha)>0$ such that for $L\geq L_0$,
 \[\mathbb{P}[\exists x\in B(n),y\in\mathbb{Z}^d:||x-y||_{\infty}\geq n^{\frac{d}{d+1}},x\overset{\mathcal{L}_{\alpha}^{\leq L}}{\longleftrightarrow}y,x\overset{\mathcal{L}_{\alpha}^{\leq L}}{\centernot{\longleftrightarrow}}\infty]\leq Cn^de^{-cn^{\frac{d}{d+1}}}.\]
 We group the summation into $3^d$ summations as follows: $\sum\limits_{\vec{j}}S(\vec{j})=\sum\limits_{\vec{k}\in\{0,1,2\}^d}S^{(\vec{k})}$ where $S^{(\vec{k})}=\sum\limits_{\vec{j}:3\vec{j}+\vec{k}\in J}S(3\vec{j}+\vec{k})$. By union bounds, it suffices to prove that for each $\vec{k}$, $\forall\epsilon>0$, there exist $L_0=L_0(d,\alpha,\epsilon)$, $c=c(d,\alpha,\epsilon,L)>0$ and $C=C(d,\alpha,\epsilon,L)<\infty$ such that $\forall L\geq L_0$,
 \begin{equation}
  \mathbb{P}[S^{(\vec{k})}>\epsilon n^d]\leq C\exp(-cn^{\frac{d}{d+1}}).
 \end{equation}
 The dependences of $c$ and $C$ on $L$ is due to the assumption that $m=\lceil n^{\frac{d}{d+1}}\rceil\geq L$. Note that for each $\vec{k}$, conditionally on $\mathcal{L}_{\alpha}^{>L}$, the variables $(S(3\vec{j}+\vec{k}))_{\vec{j}}$ are independent since they dependent on disjoint loop ensembles within a Poisson point process of loops. Also, note that $\#J\leq 2^d(n/m+1)^{d}$ and that $S(\vec{j})\leq m^{d}$. By Hoeffding's inequality, $\exists c=c(d,\epsilon)$ such that
 \[\mathbb{P}\left[S^{(\vec{k})}-\mathbb{E}\big[S^{(\vec{k})}|\mathcal{L}_{\alpha}^{>L}\big]>\epsilon n^d\Big|\mathcal{L}_{\alpha}^{>L}\right]\leq \exp(-cn^{\frac{d}{d+1}}).\]
 By the independence between $\mathcal{L}_{\alpha}^{>L}$ and $\mathcal{L}_{\alpha}^{\leq L}$, by Lemma~\ref{lem: mean cluster size bound}, for $m>L$, $\exists C=C(d,\alpha)<\infty$ such that
 \[\mathbb{E}\big[S^{(\vec{k})}|\mathcal{L}_{\alpha}^{>L}\big]\leq C\sum\limits_{y\in B(n+m)}1_{\{\exists\ell\in\mathcal{L}_{\alpha}^{>L}:y\in\ell\}}.\]
 Finally, the result follows from Lemma~\ref{lem: vertices visited by large loops}.

\subsection{Proof of Lemma~\ref{lem: proba HKaL}}\label{sect: proof of lem proba HKaL}
We take $L$ large enough such that the truncated model $\mathcal{L}_{\alpha}^{\leq L}$ percolates and Corollary~\ref{cor: exponential tail}, Lemma~\ref{lem: big loop plus small finite cluster} and Lemma~\ref{lem: local connectness} are applicable.

By the definition of $\mathcal{H}_{K,s}^{\alpha,L}$, we bound the complement of $\mathcal{H}_{K,s}^{\alpha,L}$ as follows:
\begin{align*}
 \mathbb{P}\left[\left(\mathcal{H}_{K,s}^{\alpha,L}\right)^c\right]\leq &\mathbb{P}\left[\cup_{x\in [-2,K+2)^d}\{x\text{ is }s\text{-bad w.r.t. }\mathcal{L}_{\alpha}^{\leq L}\}\right]\\
 &+\mathbb{P}\left[
 \exists x,y\in\mathcal{S}_{L_s}\cap [0,KL_s)^d\text{ with }||x-y||_{\infty}\leq L_s\text{ such that }x\overset{B(x,2L_s),\mathcal{L}_{\alpha}}{\centernot{\longleftrightarrow}}y\right]\\
 &+\mathbb{P}\left[\begin{array}{l}
                    \exists \vec{k}\in\{-2,\ldots,K+1\}^d\text{ such that }\\
                    \#\big\{x\in B^{(L_s)}(\vec{k}) : x\overset{\mathcal{L}_{\alpha}^{\leq L}}{\centernot{\longleftrightarrow}}\infty,\exists \ell\in\mathcal{L}_{\alpha}^{>L}\text{ such that }x\overset{\mathcal{L}_{\alpha}^{\leq L}}{\longleftrightarrow}\ell\big\}\geq \frac{\theta(\alpha)}{100} L_s^d
                   \end{array}\right].
\end{align*}
By taking $\epsilon_P=\infty$ in \cite[Lemma~3.2]{Sapozhnikov2014}, there exist $C(d,\alpha,L)<\infty$ and $C'=C'(d,\alpha,l_0,L)$ such that for $r_0\geq C$, $l_0\geq Cr_0$ and $L_0>C'$, we have that
 \[\mathbb{P}\left[\cup_{x\in [-2,K+2)^d}\{x\text{ is }s\text{-bad w.r.t. }\mathcal{L}_{\alpha}^{\leq L}\}\right]\leq 2(K+4)^d2^{-2^s}.\]
By Lemma~\ref{lem: big loop plus small finite cluster}, there exist $c=c(d,\alpha)>0$ and $C=C(d,\alpha)<\infty$ such that
\begin{multline*}
\mathbb{P}\left[\begin{array}{l}
                    \exists \vec{k}\in\{-2,\ldots,K+1\}^d\text{ such that }\\
                    \#\big\{x\in B^{(L_s)}(\vec{k}) : x\overset{\mathcal{L}_{\alpha}^{\leq L}}{\centernot{\longleftrightarrow}}\infty,\exists \ell\in\mathcal{L}_{\alpha}^{>L}\text{ such that }x\overset{\mathcal{L}_{\alpha}^{\leq L}}{\longleftrightarrow}\ell\big\}\geq \frac{\theta(\alpha)}{100} L_s^d
                   \end{array}\right]\\
\leq C(K+4)^d\exp\left(-cL_s^{\frac{d}{d+1}}\right).
\end{multline*}
By Corollary~\ref{cor: exponential tail}, there exist $c=c(d,\alpha)>0$ and $C=C(d,\alpha)<\infty$ such that
 \[\mathbb{P}\left[\begin{array}{l}
 \exists x,y\in\mathcal{S}_{L_s}\cap [0,KL_s)^d\text{ with }||x-y||_{\infty}\leq L_s,\\
 \text{ such that }x\notin\mathcal{S}_{\infty}\text{ or }y\notin\mathcal{S}_{\infty}
 \end{array}\right]\leq C(KL_s)^dL_s^d e^{-cL_s},\]
 where $(KL_s)^dL_s^d$ comes from the combinatorial complexity of $x$ and $y$. By Lemma~\ref{lem: local connectness}, there exist $c=c(d,\alpha)>0$ and $C=C(d,\alpha)<\infty$ such that
 \[\mathbb{P}\left[\begin{array}{l}
 \exists x,y\in\mathcal{S}_{\infty}\cap [0,KL_s)^d\text{ with }||x-y||_{\infty}\leq L_s\\
 \text{such that }x\overset{B(x,2L_s),\mathcal{L}_{\alpha}}{\centernot{\longleftrightarrow}}y
 \end{array}\right]\leq C(KL_s)^dL_s^de^{-cL_s}.\]
Hence,
 \[\mathbb{P}\left[\begin{array}{l}
 \exists x,y\in\mathcal{S}_{L_s}\cap [0,KL_s)^d\text{ with }||x-y||_{\infty}\leq L_s\\
 \text{such that }x\overset{B(x,2L_s),\mathcal{L}_{\alpha}}{\centernot{\longleftrightarrow}}y
 \end{array}\right]\leq C(KL_s)^dL_s^de^{-cL_s}\]
Therefore, Lemma~\ref{lem: proba HKaL} follows.

\subsection{Proof of Lemma~\ref{lem: isoperimetric}}\label{sect: proof of lem isoperimetric}
 By translation invariance of $\mathcal{L}_{\alpha}$, we give the proof for $x_s=0$. Choose $L$ large enough such that $\theta^{(L)}(\alpha)>\frac{99}{100}\theta(\alpha)$. Assume $\mathcal{H}_{K,s}^{\alpha,L}$ occurs. We state a key lemma on $\mathcal{Q}_{K,s,0}^{\leq L}(x_s)$ which will be used to deduce Lemma~\ref{lem: isoperimetric}.
 \begin{lemma}\label{lem: perforated}[Weaker version of \cite[Lemma~2.4,Corollary~2.5]{Sapozhnikov2014}]
 Suppose that all the vertices in $\mathbb{G}_{s}\cap [-2L_s,(K+2)L_s)^d$ are $s$-good w.r.t. $\mathcal{L}_{\alpha}^{\leq L}$ for $L<\infty$. Then, for $z,z'\in\mathbb{G}_{s}\cap [-2L_s,(K+2)L_s)^d$ such that $|z-z'|_1=L_s$, we have that
 \begin{equation}\label{eq: perforated volume lower bound}
  \mathcal{Q}_{K,s,0}^{\leq L}\cap (z+[0,L_s)^d)\geq \left(\frac{L_s}{L_0}\right)^d\prod_{i=1}^{\infty}\left(1-\left(\frac{4r_i}{l_i}\right)^d\right)\geq \left(\frac{L_s}{L_0}\right)^d(1-10^{-12}\theta^{(L)}(\alpha))
 \end{equation}
 and $\mathcal{Q}_{K,s,0}^{\leq L}\cap ((z+[0,L_s)^d)\cup(z'+[0,L_s)^d))$ is a connected set in $\mathbb{G}_0$. There exists $c=c(d,\alpha,L)>0$ such that for $A\subset \mathcal{Q}_{K,s,0}^{\leq L}$ with $\min(c(K+4)^d(L_s/L_0)^d,(L_s/L_0)^{d^2})\leq \#A\leq\frac{1}{2}\cdot\#\mathcal{Q}_{K,s,0}^{\leq L}$, we have that
 \[\partial_{\mathcal{Q}_{K,s,0}^{\leq L}}A\geq c\cdot\#A^{\frac{d-1}{d}}.\]
 \end{lemma}
 Next, we explain the way to deduce similar properties for $\widetilde{\mathcal{C}}_{K,s,L_0}^{\leq \infty}$ from Lemma~\ref{lem: perforated} when $\mathcal{H}^{\alpha,L}_{K,s}$ occurs. Although $\cup_{x\in\mathcal{Q}_{K,s,0}^{\leq L}}\mathcal{C}_x^{\leq L}$ itself is not necessarily connected, its vertices are connected inside $\mathcal{F}_{K,s,L_0}^{\leq L}$, where
 \[\mathcal{F}_{K,s,L_0}^{\leq L}\overset{\text{def}}{=}\bigcup_{\substack{x,y\in\mathcal{Q}_{K,s,0}^{\leq L}\\|x-y|_1=L_0}}\left\{z: \begin{array}{l}z\overset{(x+[0,L_0)^d)\cup(y+[0,L_0)^d),\mathcal{L}_{\alpha}^{\leq L}}{\longleftrightarrow}\mathcal{C}_x^{\leq L}\\
 \text{and }z\overset{(x+[0,L_0)^d)\cup(y+[0,L_0)^d),\mathcal{L}_{\alpha}^{\leq L}}{\longleftrightarrow}\mathcal{C}_y^{\leq L}\end{array}\right\}.\]
 By definition of $0$-good vertices, for $x\in\mathcal{Q}_{K,s,0}^{\leq L}$, we have that $\frac{9}{10}\theta^{(L)}(\alpha)L_0^d\leq\#\mathcal{C}_x^{\leq L}\leq\frac{11}{10}\theta^{(L)}(\alpha)L_0^d$ and that $\mathcal{S}_{L_0}^{\leq L}\cap(x+[0,L_0)^d)\leq\frac{11}{10}\theta^{(L)}(\alpha)L_0^d$. Hence, by \eqref{eq: perforated volume lower bound}, we see that
 \begin{equation}\label{eq: sapo eq 3.3}
  \#\bigcup_{x\in [0,K)^d\cap\mathcal{Q}_{K,s,0}^{\leq L}}\mathcal{C}_x^{\leq L}\geq \frac{9}{10}\theta^{(L)}(\alpha)\prod\limits_{i=0}^{\infty}\left(1-\left(\frac{4r_i}{l_i}\right)^d\right)L_s^d\geq\frac{9}{10}\cdot\frac{99}{100}(1-10^{-12})\theta(\alpha)K^dL_s^d,
 \end{equation}
 and that
 \begin{align}\label{eq: sapo eq 3.4}
  \#\left(\mathcal{S}_{L_0}^{\leq L}\cap [0,KL_s)^d\right)\leq & L_0^d((KL_s/L_0)^d-\#\mathcal{Q}_{K,s,0}^{\leq L}\cap [0,KL_s)^d))\notag\\
  &+\frac{11}{10}\theta^{(L)}(\alpha)L_0^d\cdot\#\mathcal{Q}_{K,s,0}^{\leq L}\cap [0,KL_s)^d)\notag\\
  \leq & K^dL_s^d\left(\frac{11}{10}\theta^{(L)}(\alpha)\prod\limits_{i=0}^{\infty}\left(1-\left(\frac{4r_i}{l_i}\right)^d\right)+1-\prod\limits_{i=0}^{\infty}\left(1-\left(\frac{4r_i}{l_i}\right)^d\right)\right)\notag\\
  \leq & (11/10+10^{-12})\theta(\alpha)\cdot K^dL_s^d.
 \end{align}
 By \eqref{eq: sapo eq 3.4} and the property $c)$ in the definition of $\mathcal{H}_{K,s}^{\alpha,L}$, we see that
 \begin{equation}\label{eq: variant sapo eq 3.4}
  \#\left(\mathcal{S}_{L_0}^{\leq \infty}\cap [0,KL_s)^d\right)\leq (11/10+10^{-12}+1/100)\theta(\alpha)\cdot K^dL_s^d.
 \end{equation}
 Note that \eqref{eq: sapo eq 3.3} and \eqref{eq: variant sapo eq 3.4} imply that \[\#\cup_{x\in(x_s/L_s+[0,K)^d)\cap\mathcal{Q}_{K,s,0}^{\leq L}}\mathcal{C}_x^{\leq L}\geq\frac{1}{2}\cdot \#\left(\mathcal{S}_{L_0}^{\leq \infty}\cap [0,KL_s)^d)\right)\geq\frac{1}{2}\cdot\#\mathcal{C}_{K,s,L_0}^{\leq \infty}.\]
 Hence, $\cup_{x\in(x_s/L_s+[0,K)^d)\cap\mathcal{Q}_{K,s,0}^{\leq L}}\mathcal{C}_x^{\leq L}\subset\mathcal{C}_{K,s,L_0}^{\leq \infty}$ and $\cup_{x\in\mathcal{Q}_{K,s,0}^{\leq L}}\mathcal{C}_x^{\leq L}$ is a subset of $\widetilde{\mathcal{C}}_{K,s,L_0}^{\leq \infty}$.
 Similar to the proof of \eqref{eq: sapo eq 3.3}, we get \eqref{eq: li2}.

 For the connectness properties in Lemma~\ref{lem: isoperimetric}, we also follow the strategy of \cite{Sapozhnikov2014}. By definition of $\widetilde{\mathcal{C}}_{K,s,L_0}^{\leq \infty}$, for two vertices $x,y\in\widetilde{\mathcal{C}}_{K,s,L_0}^{\leq L}$, there exist $x',y'\in\mathcal{C}_{K,s,L_0}^{\leq L}$ such that $x\overset{B(x',2L_s),\mathcal{L}_{\alpha}^{\leq L}}{\longleftrightarrow}x'$ and $y\overset{B(y',2L_s),\mathcal{L}_{\alpha}^{\leq L}}{\longleftrightarrow}y'$. Next, by part b) in the definition of $\mathcal{H}_{K,s}^{\alpha,L}$, there exist $x'',y''\in\cup_{x\in\mathcal{Q}_{K,s,0}^{\leq L}}\mathcal{C}_x^{\leq L}$ such that $x'\overset{B(x'',2L_s),\mathcal{L}_{\alpha}^{\leq L}}{\longleftrightarrow}x''$ and $y'\overset{B(y'',2L_s),\mathcal{L}_{\alpha}^{\leq L}}{\longleftrightarrow}y''$. Let $u,v\in\mathbb{G}_s$ such that $x''\in (u+[0,L_s)^d)\cap\mathcal{Q}_{K,s,0}^{\leq L}$ and $y''\in(v+[0,L_s)^d)\cap\mathcal{Q}_{K,s,0}^{\leq L}$. Then, by Lemma~\ref{lem: perforated}, for any nearest neighbor path $z_0=u,z_1,\ldots,z_n=v$ in $\mathbb{G}_s$ connecting $u$ and $v$, $\cup_{i=0}^{n}(z_i+[0,L_s)^d)\cap\mathcal{Q}_{K,s,0}^{\leq L}$ is a connected set. Thus, the connectness properties in Lemma~\ref{lem: isoperimetric} follows from the definition of the $0$-good vertices.

 For the isoperimetric inequalities, it is also a consequence of Lemma~\ref{lem: perforated} and we follow closely the proof of \cite[Theorem~2.14]{Sapozhnikov2014}. We take $A\subset\widetilde{\mathcal{C}}_{K,s,L_0}^{\leq\infty}$ that $2^d\cdot100^{d^2}L_s^{d^2}\leq\#A\leq\frac{1}{2}\cdot\#\widetilde{\mathcal{C}}_{K,s,L_0}^{\leq\infty}$. (For a general $\delta\in(0,1)$, it suffices to consider $\widetilde{\mathcal{C}}_{K,s,L_0}^{\leq \infty}\setminus A$ if necessary.) We define two subsets:
 \begin{align*}
  \mathcal{K}=&\{x\in\mathbb{G}_0:(x+[0,L_0)^d)\cap A\neq\emptyset\}\\
  \mathcal{K}'=&\{x\in\mathbb{G}_0:(x+[0,L_0)^d)\cap A\neq\emptyset, (x+[0,L_0)^d)\cap A=(x+[0,L_0)^d)\cap\widetilde{\mathcal{C}}_{K,s,L_0}^{\leq\infty}\}.
 \end{align*}
 Note that $\#\mathcal{K}\geq \#A/L_0^d$. By the connection property we have just proved, for each $x\in\mathcal{K}\setminus\mathcal{K'}$, there exist two vertices $z\in A,z'\in \widetilde{\mathcal{C}}_{K,s,L_0}^{\leq \infty}\setminus A$ such that $z'$ is connected to $z$ inside $B(z,17L_s)\cap\widetilde{\mathcal{C}}_{K,s,L_0}^{\leq \infty}$. This implies that $\#\partial_{\widetilde{\mathcal{C}}_{K,s,L_0}^{\leq \infty}} A\geq \#(\mathcal{K}\setminus\mathcal{K'})/(100L_s/L_0)^d$. If $\#\mathcal{K}'\leq \frac{1}{2}\cdot\#\mathcal{K}$, then
 \[\#\partial_{\widetilde{\mathcal{C}}_{K,s,L_0}^{\leq \infty}} A\geq \#A/(2(100L_s)^d)\geq \#A^{\frac{d-1}{d}}.\]
 It remains to consider the case $\#\mathcal{K}'\geq \frac{1}{2}\cdot\#\mathcal{K}\geq\frac{1}{2}\cdot\#A/L_0^d$. We define a subset $\mathcal{A}$ of $\mathcal{Q}_{K,s,0}^{\leq L}$ as follows:
 \[\mathcal{A}\overset{\text{def}}{=}\{x\in \mathcal{Q}_{K,s,0}^{\leq L}:\mathcal{C}_x^{\leq L}\subset A\}.\]
 Then, $\#\mathcal{A}\geq \#\mathcal{K}'\geq \frac{1}{2}\cdot\#\mathcal{K}\geq\frac{1}{2}\cdot\#A/L_0^d\geq 2^{d-1}100^{d^2}(L_s/L_0)^{d^2}$. By \eqref{eq: variant sapo eq 3.4},
 \begin{equation}\label{eq: vol upper bound for tilde C}
  \#\widetilde{\mathcal{C}}_{K,s,L_0}^{\leq\infty}\leq \#\left(\mathcal{S}_{L_0}^{\leq \infty}\cap [-2L_s,(K+2)L_s)^d)\right)\leq 6/5\cdot\theta(\alpha)(K+4)^dL_s^d .
 \end{equation}
 By \eqref{eq: sapo eq 3.3}, we have that
 \begin{equation}\label{eq: vol lower bound frame}
  \#\cup_{x\in\mathcal{Q}_{K,s,L_0}^{\leq L}}\mathcal{C}_{x}^{\leq L}\geq \frac{4}{5}\theta(\alpha)(K+4)^dL_s^d.
 \end{equation}
 By \eqref{eq: vol upper bound for tilde C}, \eqref{eq: vol lower bound frame} and the definition of $0$-good vertices, we have that
 \[\frac{9}{10}\theta(\alpha)L_0^d\cdot\#\mathcal{A}\leq \#A\leq\frac{1}{2}\cdot \widetilde{\mathcal{C}}_{K,s,L_0}^{\leq\infty}\leq\frac{3}{4}\cdot\#\cup_{x\in\mathcal{Q}_{K,s,L_0}^{\leq L}}\mathcal{C}_{x}^{\leq L}\leq \frac{3}{4}\cdot\frac{11}{10}\theta(\alpha)L_0^d\cdot\#\mathcal{Q}_{K,s,L_0}^{\leq L},\]
 i.e. $\#\mathcal{A}\leq \frac{11}{12}\cdot \#\mathcal{Q}_{K,s,L_0}^{\leq L}$. By Lemma~\ref{lem: perforated} (and considering $\mathcal{Q}_{K,s,L_0}\setminus \mathcal{A}$ if necessary), we have that $\#\partial_{\mathcal{Q}_{K,s,L_0}}\mathcal{A}\geq c\cdot \#\mathcal{A}^{\frac{d-1}{d}}$, which implies that there exists $\gamma=\gamma(d,L_0)>0$ such that
 \[\#\partial_{\widetilde{\mathcal{C}}_{K,s,L_0}^{\leq\infty}}A\geq \gamma\cdot\#A^{\frac{d-1}{d}}.\]

\paragraph{Acknowledgement}
The author thanks A. Sapozhnikov for useful suggestions and fruitful discussions.
\bibliographystyle{amsalpha}
\bibliography{reference.bib}
\end{document}